\documentclass[11pt]{article}

\usepackage[utf8]{inputenc}
\usepackage{xcolor}
\usepackage{amsfonts}
\usepackage{amsmath}
\usepackage{amssymb}
\usepackage{amsthm}
\usepackage{bbm}
\usepackage{bm}
\usepackage{mathrsfs}
\usepackage{verbatim}
\usepackage{pdfsync}
\usepackage{enumitem}
\usepackage{graphicx}
\usepackage{subcaption}
\usepackage{hyperref}
\usepackage[numbers]{natbib}

\usepackage{style}

\theoremstyle{plain}

\newtheorem{theorem}{Theorem}[section]
\newtheorem{lemma}[theorem]{Lemma}
\newtheorem{proposition}[theorem]{Proposition}
\newtheorem{corollary}[theorem]{Corollary}
\newtheorem{example}[theorem]{Example}
\theoremstyle{remark}
\newtheorem{definition}[theorem]{Definition}
\newtheorem{remark}[theorem]{Remark}

\newcommand{\crl}[1]{\ensuremath{ \left\{ #1 \right\} }}
\newcommand{\edg}[1]{\ensuremath{\! \left[ #1 \right] }}
\newcommand{\brak}[1]{\ensuremath{\left( #1 \right)}}

\newcommand{\abs}[1]{\ensuremath{ \left| #1 \right| }}

\usepackage{chngcntr}
\counterwithin*{equation}{section}

\newcommand{\be}{\begin{equation}}
\newcommand{\ee}{\end{equation}}
\newcommand{\bea}{\begin{eqnarray}}
\newcommand{\eea}{\end{eqnarray}}
\newcommand{\beas}{\begin{eqnarray*}}
\newcommand{\eeas}{\end{eqnarray*}}

\newcommand{\ind}{{\perp\!\!\!\perp}}

\newcommand{\Do}{\text{\rm do}}

\newcommand{\pa}{\text{\rm pa}}

\newcommand{\cL}{\mathcal{L}}
\newcommand{\cX}{\mathcal{X}}
\newcommand{\cY}{\mathcal{Y}}
\newcommand{\cV}{\mathcal{V}}
\newcommand{\cW}{\mathcal{W}}
\newcommand{\cZ}{\mathcal{Z}}

\newcommand{\cP}{\mathcal{P}}

\newcommand{\cF}{\mathcal{F}}
\newcommand{\cS}{\mathcal{S}}
\newcommand{\bN}{\mathbb{N}}
\newcommand{\bR}{\mathbb{R}}
\newcommand{\bE}{\mathbb{E}}
\newcommand{\bP}{\mathbb{P}}

\newcommand{\R}{\mathbb{R}}

\title{Optimal transport and Wasserstein\\ distances for causal models\thanks{We are grateful to Beatrice Acciaio, Julio Backhoff-Veraguas, Daniel Bartl, Mathias Beiglböck, Songyan Hou, Nicolai Meinshausen, Alexander Neitz, Gudmund Pammer and Drago Ple\v{c}ko for interesting discussions and helpful comments.}}
\author{
  Patrick Cheridito \\  
  Department of Mathematics\\
  ETH Zurich, Switzerland 
   \And
 Stephan Eckstein \\
 Department of Mathematics\\
  University of T\"ubingen, Germany
}

\begin{document}

\maketitle

\begin{abstract}
In this paper, we introduce a variant of optimal transport adapted to the causal 
structure given by an underlying directed graph $G$. Different graph structures lead to 
different specifications of the optimal transport problem. For instance, a fully connected 
graph yields standard optimal transport, a linear graph structure corresponds to 
causal optimal transport between the distributions of 
two discrete-time stochastic processes, and an empty graph leads to a notion 
of optimal transport related to CO-OT, Gromov--Wasserstein distances and factored OT. We derive 
different characterizations of $G$-causal transport plans and introduce 
Wasserstein distances between causal models that respect the underlying graph structure.
We show that average treatment effects are continuous with respect 
to $G$-causal Wasserstein distances and small perturbations of structural causal models 
lead to small deviations in $G$-causal Wasserstein distance. We also
introduce an interpolation between causal models based on 
$G$-causal Wasserstein distance and compare it to standard Wasserstein interpolation.\\[2mm]
\end{abstract}

\section{Introduction}

Originally, optimal transport problems were introduced by Monge \citep{monge1781memoire} and, in a 
more general form, by Kantorovich \citep{kantorovich1942translocation} to study the most efficient
way to transport and allocate resources.
But in addition to this immediate application, optimal transport theory has also 
lead to the the notion of Wasserstein distance \citep{kantorovich1960mathematical, vallender1974calculation, vaserstein1969markov}, which defines a metric
on probability distributions.
Over the years, optimal transport has found applications in different areas of economics 
\citep{chiappori2010hedonic, galichon2018optimal, pflug20121}, 
probability theory \citep{Rachev1985, RachRuesch1998},
statistics \citep{fournier2015rate, ghosal2022multivariate, niles2022minimax}, 
differential geometry \citep{de2014monge,figalli2010mass,Sud76}, 
robust optimization \citep{blanchet2019quantifying,mohajerin2018data, ConOpt2017}, 
machine learning and data science \citep{arjovsky2017wasserstein,cuturi2014fast, peyre2019computational,IPM2012}, just to name a few.
At the same time, various variants and extensions of optimal transport have emerged,
like multi-marginal versions \citep{agueh2011barycenters, embrechts2013model, pass2015multi}, optimal transport with additional constraints \citep{beiglbock2016problem, CKT, CKPS, de2020minmax,korman2015optimal,nutz2022directional}, optimal transport between measures of different masses \citep{chizat2018unbalanced, unbalancedThibault}, relaxations \citep{bonneel2019spot, LiLin2021} 
and regularizations \citep{cuturi2013sinkhorn,lorenz2021quadratically}. On the other hand,
causal research examines causal relationships between different events; see e.g.
\citep{holland1986statistics, pearl2009causality, peters2016causal, scholkopf2022causality}.
In probabilistic theories of causality, the underlying causal structure is typically 
described by a directed graph, leading to graphical causal models \citep{bongers2021foundations}.

In this paper we introduce a version of optimal transport that is adapted to the 
causal structure given by an underlying directed graph and investigate 
corresponding Wasserstein distances between different causal models. 
Our main goal is to obtain a version of the optimal transport problem which yields topological and geometric properties consistent with the structure of the underlying directed graph.
We build on existing literature studying this question for distributions of stochastic processes such as
\citep{backhoff2019stability,backhoff2017causal,bartl2021wasserstein,lassalle2018causal}. This string of literature demonstrates different advantages of
incorporating the temporal structure into the definition of optimal transport distances between distributions of stochastic processes compared to using standard Wasserstein distances, which are agnostic to this structure. For instance, it improves the stability of dynamic optimization problems that are unstable under standard Wasserstein distances
(see \citep{backhoff2019stability}), and it leads to geometric structures compatible with certain properties of stochastic processes (e.g. it has been shown in \citep{bartl2021wasserstein} that sets of martingales are geodesically convex). In this paper, we show that it has similar benefits if a suitably adapted concept of optimal transport is used for general causal
structures. For such a version of optimal transport, the structure of the 
underlying graph $G$ determines the exact specification of the optimal transport problem.
A fully connected graph corresponds to standard optimal transport, which consists
in finding an optimal transport plan transforming a given distribution 
into another one. Missing edges translate into additional constraints, 
which increase the $G$-causal Wasserstein distance and therefore,
lead to a finer topology on the corresponding set of causal models. For instance,
a linear graph requires transport plans to respect the temporal structure. This results in 
causal transport problems between distributions of stochastic processes, which have been studied by e.g. \citep{backhoff2019stability,backhoff2017causal,bartl2021wasserstein,lassalle2018causal}.
An empty graph restricts transport plans the most and is related to CO-OT problems \citep{titouan2020co}, the Gromov--Wasserstein distance \citep{memoli2011gromov} and factored OT \citep{tran2021factored}.
For a general directed graph $G$, we call the corresponding transport plans $G$-causal and demonstrate their suitability for downstream tasks which depend on the information structure given by $G$. For instance, we show that while one of the most fundamental causal statistics, the average treatment effect, is not continuous in the standard Wasserstein metric, it becomes Lipschitz continuous if distances between distributions are measured with a $G$-causal Wasserstein distance resulting from $G$-causal optimal transport. Furthermore, we demonstrate that under suitable assumptions, $G$-causal Wasserstein interpolation between two $G$-compatible distributions preserves $G$-compatibility, which in general, is not true for standard Wasserstein interpolation.


The remainder of this paper is organized as follows: In Section \ref{sec:notation}
we introduce the setup and provide our definition of $G$-causal transport plans together 
with the corresponding optimal transport problems in Definitions \ref{def:causal} and \ref{def:causalot}. 
In Section \ref{subsec:specialcases} we showcase the relation to existing concepts of 
optimal transport from the literature. In Section \ref{sec:DAG} we 
focus on $G$-causal transport maps for directed acyclic graphs. We derive alternative 
characterizations of $G$-causal transport plans in Theorem \ref{thm:main} and 
Corollary \ref{cor:main} and study the structure of sets of 
$G$-causal transport plans in Proposition \ref{prop:basic_prop}. In Section \ref{sec:Wasserstein} 
we introduce Wasserstein distances between causal models that respect the structure of the 
underlying graph $G$. We show that they satisfy all properties of a metric except for 
the triangle inequality in Proposition \ref{prop:wassersteinprop}. In Section \ref{subsec:ate}, 
we prove that, under suitable assumptions, average treatment effects, which estimate the
causal effect of a treatment or intervention, are Lipschitz continuous with respect to 
$G$-causal Wasserstein distance under a change of the underlying probability model
while they are not continuous for standard Wasserstein distances. As a consequence, bounds on the error resulting from estimating average treatment effects from data can be obtained by controlling the $G$-causal Wasserstein estimation error.
On the other hand, we show in Section \ref{subsec:closeclose} that small perturbations of a 
structural causal model correspond to small deviations in $G$-causal Wasserstein 
distance. In particular, convergence of structural models implies convergence of the corresponding probability models with respect to G-causal Wasserstein distance. Finally, in Section \ref{subsec:geodesic} we study a Wasserstein interpolation between causal 
models that respects the causal structure and compare it to standard 
Wasserstein interpolation.

\section{Notation and definitions}
\label{sec:notation}

\subsection{Causal structures described by directed graphs}

We consider a finite set $V = \crl{1, \dots, n}$ 
for some $n \in \bN$, endowed with a causal structure
given by a set of {\sl directed edges} $E \subseteq \{(i,j) \in V^2 : i \neq j\}$. 
This turns $G = (V,E)$ into a {\sl directed graph}. For $j \in V$, we denote by 
$\pa_j$ the parents of $j$, given by 
\[
\pa_j := \{i \in V : (i,j) \in E\},
\]
and for a subset $A \subseteq V$, we define
\[
\pa_A := \bigcup_{j \in A} \pa_j.
\]
We order both sets $\pa_j$ and $\pa_A$ according to the natural ranking of the elements of
$V = \{1, \dots, n\}$. 

\subsection{Spaces and measures}

We consider two product spaces of the form $\cX = \cX_1\times \dots \times \cX_n$ and 
$\cY = \cY_1 \times \dots \times \cY_n$ for non-empty 
Polish spaces ${\cal X}_i, {\cal Y}_i$, $i = 1, \dots, n$.
We endow $\cX$, $\cY$ and $\cX \times \cY$ 
with the product topologies, which makes them again Polish; see e.g. 
\citep[Chapter IX, §6.1, Proposition I]{bourbakitopology}. By
$\cP(\cX)$, $\cP(\cY)$ and $\cP(\cX \times \cY)$ we denote the set of all Borel probability 
measures on $\cX$, $\cY$ and $\cX \times \cY$, respectively. 
If we write $X \sim \mu$ for a measure $\mu \in \mathcal{P}(\cX)$, we mean 
that $X$ is an $\cX$-valued random variable with distribution $\mu$ defined on 
some underlying probability space.
Moreover, we use the notation $X_{1:i}$ to denote the tuple $(X_1, \dots, X_i)$ for 
$i \in V = \{1, \dots, n\}$, and similarly, $X_A$ for any ordered subset $A$ of $V = \{1, \dots, n\}$.
$X_{1:0}$ as well as $X_{\pa_i}$ in cases where $\pa_i$ is empty, 
will be understood as empty tuple.

\subsection{Transport plans and $G$-causal couplings}

A transport map between two probability measures $\mu \in \cP(\cX)$ and $\nu \in \cP(\cY)$
in the sense of Monge \citep{monge1781memoire} is a measurable mapping 
$T \colon \cX \to \cY$ such that $T_{\#} \mu = \nu$, where $T_{\#} \mu$ 
denotes the push-forward $\mu \circ T^{-1}$ of $\mu$ along $T$.
A Monge transport map does not always exist, e.g. if $\mu$ is a Dirac measure and $\nu$ is not.
Therefore, Kantorovich \citep{kantorovich1942translocation} extended the set of transport plans by considering couplings 
between $\mu$ and $\nu$, which are probability measures $\pi \in \cP(\cX \times \cY)$ 
with marginals $\mu$ and $\nu$. The set $\Pi(\mu, \nu)$ of all couplings 
between $\mu$ and $\nu$ always contains the product measure $\mu \otimes \nu$, 
and therefore, is non-empty. Moreover, for every Monge transport map $T \colon \cX \to \cY$ 
between $\mu$ and $\nu$, the distribution of the mapping 
$(\mbox{id}, T) \colon \cX \to \cX \times \cY$ under $\mu$ is a coupling in $\Pi(\mu, \nu)$.
On the other hand, every coupling $\pi \in \Pi(\mu, \nu)$ can be realized as a {\sl randomized
transport map}. Indeed, since $\cY$ is Polish, $\pi$ admits a disintegration of the form
$\pi(dx, dy) = \mu(dx) \pi(dy \mid x)$, where 
$\pi(dy \mid x)$ is a regular version of the conditional 
distribution of $y \in \cY$ given $x \in \cX$ with respect to $\pi$ ; see e.g. 
\citep[Theorem 3.4]{kal3rd}. Therefore, one obtains from the kernel 
representation \citep[Lemma 4.22]{kal3rd} that
$\pi$ can be realized as the distribution of a pair of 
random variables $(X,Y)$ such that $X \sim \mu$ and $Y$ is of the form 
$Y = g(X,U)$ for a measurable map $g \colon \cX \times \R \to \cY$ and 
an $\R$-valued random variable $U$ defined on the same 
probability space as and independent of $X$.

To motivate our definition of $G$-causal couplings, we first consider the case of 
Monge transport maps $T \colon \cX = \cX_1 \times \dots \times \cX_n \to 
\cY = \cY_1 \times \dots \cY_n$. Intuitively, one might want to
define $G$-causality of $T$ by saying that $T$ is
$G$-causal if it is of the form $T= (T_1, \dots, T_n)$, where
\be \label{T1}
T_i(x) = g_i(x_i, x_{\pa_i}) \quad \mbox{for all } x \in \cX 
\mbox{ and } i = 1, \dots, n,
\ee
for measurable mappings $g_i \colon \cX_i \times \cX_{\pa_i} \to \cY_i$,
$i = 1, \dots, n$. More generally, one could call $T = (T_1, \dots, T_n)$ $G$-causal if 
 \be \label{T2}
T_i(x) = g_i(x_i, x_{\pa_i}, T_{\pa_i}(x)) \quad \mbox{for all } x \in \cX
\mbox{ and } i = 1, \dots, n,
\ee
for measurable mappings $g_i \colon \cX_i \times \cX_{\pa_i} \times \cY_{\pa_i} \to \cY_i$,
$i = 1, \dots, n$. It is easy to see that in the following three cases, conditions 
\eqref{T1} and \eqref{T2} are equivalent:

\begin{itemize}
	\item if the graph $G= (V,E)$ is fully connected, that is, 
	$E = \{(i,j) \in V^2 : i \neq j\}$,
	\item 
	if the graph $G= (V,E)$ is empty, that is, $E = \emptyset$, or
	\item
	if the graph $G= (V,E)$ has the linear structure $E = \{(i,j) \in V^2 : i < j\}$.
\end{itemize}
On the other hand, 
\begin{itemize}
	\item if the graph $G= (V,E)$ has the Markovian structure 
	$E = \{(i-1, i) : 2 \le i \le n\}$,
\end{itemize}
\eqref{T2} is slightly more flexible than \eqref{T1}.
In general, \eqref{T1} is too restrictive for our purposes. For instance, a generalization of \eqref{T1} to randomized transport maps would still not guarantee the existence of 
G-causal couplings between any two probability measures $\mu$ and $\nu$ that are compatible with respect to a 
directed acyclic graph G in the sense of Definition \ref{def:Gcomp} below. Therefore, we use \eqref{T2} to 
define $G$-causal Monge transport maps, which leads to the following general definition.

\begin{definition}[$G$-causal couplings] \label{def:causal} 
	Let $G = (V,E)$ be a directed graph. We call a coupling $\pi \in \Pi(\mu, \nu)$
	between two measures $\mu \in \mathcal{P}(\cX)$ and $\nu \in \mathcal{P}(\cY)$
	$G$-causal, if there exists a pair of random variables $(X, Y) \sim \pi$ together with 
	measurable mappings
	\[
	g_i \colon \cX_{i} \times \cX_{\pa_i} \times \cY_{\pa_i} \times \mathbb{R} \rightarrow \cY_i,
	\quad i = 1, \dots, n,
	\]
	and $\R$-valued random variables $U_1, \dots, U_n$ defined on the same 
	probability space as $(X,Y)$ such that $X,U_1, \dots, U_n$ are independent and
	\begin{equation}
	\label{condcausal} 
	Y_i = g_i(X_{i}, X_{\pa_i}, Y_{\pa_i}, U_i) \quad \mbox{for all } i = 1, \dots, n.
	\end{equation}
	We call $\pi$ $G$-bicausal, if $\pi$ and $\tilde\pi \sim (Y, X)$ are $G$-causal. 
	We denote the set of $G$-causal and $G$-bicausal couplings between $\mu$ and $\nu$ by $\Pi_{G}(\mu, \nu)$ and $\Pi^{\rm bc}_{G}(\mu, \nu)$, respectively.
\end{definition}

$\Pi^{\rm bc}_G(\mu, \nu)$ is a symmetrized version of $\Pi_G(\mu, \nu)$
consisting of transport plans that are $G$-causal in both directions. 

In Theorem \ref{thm:main} and Corollary \ref{cor:main} below we give alternative characterizations of $G$-causal and $G$-bicausal transport plans in the case where $G$ is a directed acyclic graph. Proposition \ref{prop:basic_prop} shows that $\Pi_{G}(\mu, \nu)$ and $\Pi_G^{\rm bc}(\mu, \nu)$ are non-empty for all probability measures $\mu$ and $\nu$ respecting the structure of G. 
The following transport problems can be 
defined for arbitrary directed graphs $G$.

\begin{definition}[$G$-causal optimal transport]
	\label{def:causalot}
	Let $G = (V,E)$ be a directed graph. For $\mu \in \cP(\cX)$, $\nu \in \cP(\cY)$ 
	and a measurable cost function $c : \cX \times \cY \rightarrow [0,\infty)$,
	we define the $G$-causal and $G$-bicausal optimal transport problems as
	\be \label{ot}
	\inf_{\pi \in \Pi_{G}(\mu, \nu)} \int_{\cX \times \cY} c \; d\pi
	\qquad \mbox{and} \qquad
	\inf_{\pi \in \Pi^{\rm bc}_{G}(\mu, \nu)} \int_{\cX \times \cY} c \; d\pi,
	\ee
	respectively.
\end{definition}

\subsection{Special cases and related concepts}\label{subsec:specialcases}

\subsubsection{Fully connected graphs and standard optimal transport}
\label{ss:full}

If $E$ equals the set $\{(i,j) \in V^2 : i \neq j\}$ of all possible edges, 
a coupling $\pi \in \Pi(\mu, \nu)$ between two measures 
$\mu \in \cP(\cX)$ and $\nu \in \cP(\cY)$ satisfies Definition \ref{def:causal}
if and only if there exists a pair of random variables 
$(X, Y) \sim \pi$ together with measurable mappings
\[
g_i \colon \cX \times \cY_{V \setminus \{i\}} \times \mathbb{R} \rightarrow \cY_i, \quad i = 1, \dots, n,
\]
and $\R$-valued random variables $U_1, \dots, U_n$ defined on the same 
probability space as $(X,Y)$ such that $X,U_1, \dots, U_n$ are independent and
\begin{equation}
\label{condfull}
Y_i = g_i(X, Y_{V \setminus \{i\}}, U_i) \quad \mbox{for all } i = 1, \dots, n.
\end{equation}
Since $\cX$ and $\cY_i$, $i = 1, \dots, n$, are Polish spaces,  
any coupling $\pi \in \Pi(\mu, \nu)$ admits an iterated disintegration of the form 
\[
\pi(dx, dy) = \mu(dx) \pi(dy_1 \mid x)
\pi(dy_2 \mid x, y_1) \cdots \pi(dy_n \mid x, y_{1:n-1})
\]
for regular conditional distributions $\pi(dy_1 \mid x)$ and $\pi(dy_i \mid x, y_{1:i-1})$,
$i = 2, \dots, n$; see e.g. \citep[Theorem 3.4]{kal3rd}. So, it follows from the kernel representation 
\citep[Lemma 4.22]{kal3rd} that $\pi$ has a realization of the form \eqref{condfull}, which shows 
that 
\[
\Pi_G(\mu, \nu) = \Pi^{\rm bc}_G(\mu, \nu) = \Pi(\mu, \nu).
\]
That is, the two optimal transport problems in \eqref{ot} reduce to standard optimal transport.

\subsubsection{Empty graphs, CO-OT and Gromov--Wasserstein distances}
\label{ss:empty}

If the graph has no edges, that is, $E = \emptyset$, then a coupling 
$\pi \in \Pi(\mu, \nu)$ satisfies Definition 
\ref{def:causal} if and only if there exists a pair of random variables $(X, Y) \sim \pi$ together with 
measurable mappings
\[
g_i \colon \cX_i \times \mathbb{R} \rightarrow \cY_i, \quad i = 1, \dots, n,
\]
and $\R$-valued random variables $U_1, \dots, U_n$ defined on the same 
probability space as $(X,Y)$ such that $X,U_1, \dots, U_n$ are independent and
\be \label{condempty}
Y_i = g_i(X_i, U_i) \quad 
\mbox{for all } i = 1, \dots, n.
\ee
Since $\cX_i, \cY_i$, $i=1, \dots, n$, are Polish spaces, every 
measure $\pi_i \in \Pi(\mu_i, \nu_i)$ can be disintegrated as
\[
\pi_i(dx_i, dy_i) = \mu_i(dx_i) \pi_i(dy_i \mid x_i)
\]
for a regular conditional distribution 
$\pi_i(dy_i \mid x_i)$; see e.g. \citep[Theorem 3.4]{kal3rd}. Therefore, we obtain from the 
kernel representation \citep[Lemma 4.22]{kal3rd} that $\pi_i$ has a 
representation of the form \eqref{condempty}. In particular, if 
$\mu = \mu_1 \otimes \cdots \otimes \mu_n$ and
$\nu = \nu_1 \otimes \cdots \otimes \nu_n$ are product measures, one has
\[
\Pi_G(\mu, \nu) = \Pi^{\rm bc}_G(\mu, \nu) = 
\crl{\pi_1 \otimes \cdots \otimes \pi_n : 
	\pi_i \in \Pi(\mu_i, \nu_i) \mbox{ for all } i = 1, \dots, n}.
\]
This set of couplings is related to CO-OT problems \citep{titouan2020co},
the Gromov--Wasserstein distance \citep{memoli2011gromov}
and factored OT \citep{tran2021factored}.
CO-OT problems also optimize over couplings of marginals of product form.
They aim to simultaneously optimize two transport plans between 
features and labels of two datasets. Therefore,
$n=2$ is used. The Gromov--Wasserstein distance 
measures the discrepancy between two spaces. It considers $n = 2$ and marginal 
measures of the special form $\mu = \mu_1 \otimes \mu_1$, $\nu = \nu_1 \otimes \nu_1$.
Factored OT is a generalization of CO-OT to a multi-marginal setting based on couplings of the form
\[
\{\pi_1 \otimes \dots \otimes \pi_n : 
\pi_i \in \Pi \big(\mu_{i,1}, \dots, \mu_{i,n_i} \big) \text{ for } i=1, \dots, n\}.
\]

\subsubsection{Linear graph structure and transport between 
distributions of stochastic processes}
\label{ss:linear}

If the graph has the linear structure $E = \{(i,j) \in V^2 : i < j\}$, 
then a coupling $\pi \in \Pi(\mu, \nu)$ satisfies Definition \ref{def:causal}
if and only if there exists a pair of random variables 
$(X, Y) \sim \pi$ together with measurable mappings
\[
g_1 \colon \cX_1 \times \mathbb{R} \rightarrow \cY_1, \quad
g_i \colon \cX_{1:i} \times \cY_{1:i-1} \times \mathbb{R} \rightarrow \cY_i, \quad i = 2, \dots, n,
\]
and $\R$-valued random variables $U_1, \dots, U_n$ defined on the same 
probability space as $(X,Y)$ 
such that $X,U_1, \dots, U_n$ are independent and
\[
Y_1 = g_1(X_1, U_1), \quad
Y_i = g_i(X_{1:i}, Y_{1:i-1}, U_i), \quad 
 i = 2, \dots, n.
\]
If $\mu$ and $\nu$ are the distributions of two discrete-time stochastic processes with 
indices $i = 1, \dots, n$, the coupling $\pi$ can 
be viewed as a transport plan that respects the temporal structure. Such couplings, sometimes treated in more generality through the use of filtrations,
have been studied, for instance, by \citep{bartl2021wasserstein,lassalle2018causal,pflug2012distance}
and have applications in robust finance
\citep{acciaio2021cournot,backhoff2017causal,backhoff2019stability,pflug2015dynamic}
as well as machine learning \citep{xu2021quantized,xu2020cot}. For relations to 
different topologies on spaces of distributions of stochastic processes, see e.g.~\citep{aldous1981weak,backhoff2020all,hellwig1996sequential,pammer2022note}. The resulting optimal transport problems are called causal, bicausal or adapted. 

In relation to this literature, we mention that interesting extensions of the $G$-causal 
optimal transport problem may be possible by using a more flexible definition of $G$-causal couplings through more precise encodings of the information structures, similarly to how \citep{bartl2021wasserstein} extends the space of stochastic processes using filtrations.
\subsubsection{Markovian graph structure}
\label{ss:Markov}

If the graph $G = (V,E)$ has the Markovian structure $E = \{(i-1,i) : 2 \le i \le n\}$, 
a coupling $\pi \in \Pi(\mu, \nu)$ satisfies Definition \ref{def:causal}
if and only if if there exists a pair of random variables $(X, Y) \sim \pi$ together with 
measurable mappings
\[
g_1 \colon \cX_1 \times \R \to \cY_1, \quad
g_i \colon \cX_{i-1:i} \times \cY_{i-1} \times \mathbb{R} \rightarrow \cY_i, \quad i = 2, \dots, n,
\]
and $\R$-valued random variables $U_1, \dots, U_n$ defined on the same 
probability space as $(X,Y)$ such that $X,U_1, \dots, U_n$ are independent and
\[
Y_1 = g_1(X_1, U_1), \quad Y_i = g_i(X_{i-1:i}, Y_{i-1}, U_i), \quad 
i = 2, \dots, n.
\]

\subsubsection{Related concepts}
\label{ss:related}

Different variants of optimal transport have been proposed in the literature
that do not exactly fit into the framework of this paper but are related.
\citep{plecko2020fair} proposes a coupling between two populations, 
such as male and female job applicants, along a given graph with the goal of
achieving a fair embedding of one group into the other. 
\citep{kocaoglu2017causalgan} proposes a GAN framework based on 
structural assumptions similar to \eqref{condcausal}.
\citep{de2021transport} studies collections of couplings between conditional 
distributions. 

\section{$G$-causal couplings for DAGs} \label{sec:char_dag}
\label{sec:DAG}

In causal modelling it is often assumed that the causal structure is given by a 
{\sl directed acyclic graph} (DAG), which is a directed graph with no {\sl cycles}.
By reordering the vertices, a DAG can always assumed to be {\sl sorted}, which 
means that it only contains edges $(i,j)$ such that $i < j$; see e.g.
\citep{bongers2021foundations, lauritzen1996graphical, pearl2009causality}.
If $G$ is a DAG, it is possible to provide alternative characterizations of $G$-causal 
and $G$-bicausal couplings, from which simple properties of the sets of all such couplings
can be derived.

We start by recalling the notions of conditional independence and 
compatibility of a probability measure with a given DAG. 
For more details, we refer to \citep{bongers2021foundations, lauritzen1996graphical, pearl2009causality}.
For random variables $S,W,Z$ defined on the same probability space 
$(\Omega, \cF, \bP)$ with values in measurable spaces $\cS, \cW, \cZ$, we say 
$S$ is independent of $W$ given $Z$ and write $S \ind_Z W$ if 
\be \label{condindep1}
\bP[S \in A, W \in B \mid Z] = \bP[S \in A \mid Z] \, \bP[W \in B \mid Z] \quad \mbox{$\bP$-a.s.}
\ee
for all measurable subsets $A \subseteq \cS$ and $B \subseteq \cW$, which is equivalent to 
\be \label{condindep2}
\bP[S \in A \mid W, Z] = \bP[S \in A \mid Z] \quad \mbox{$\bP$-a.s.}
\ee
for all measurable subsets $A \subseteq \cS$. Since the conditional 
probabilities in \eqref{condindep1}--\eqref{condindep2} can all be 
written as $g(Z)$ or $h(W,Z)$ for measurable functions 
$g \colon \cZ \to [0,1]$ or $h \colon \cW \times \cZ \to [0,1]$ 
(see e.g. \citep[Lemma 1.14]{kal3rd}), 
\eqref{condindep1}--\eqref{condindep2} can equivalently be formulated as 
\[
\rho[A \times B \times \cZ \mid \cF_{\cZ}] 
= \rho[A \times \cW \times \cZ \mid \cF_{\cZ}] \, 
\rho[\cV \times B \times \cZ \mid \cF_{\cZ}] \quad \mbox{$\rho$-a.s.}
\]
and 
\[
\rho[A \times \cW \times \cZ \mid \cF_{\cW, \cZ}] 
= \rho[A \times \cW \times \cZ \mid \cF_{\cZ}] 
\quad \mbox{$\rho$-a.s.}, 
\]
where $\rho$ is the distribution of $(S,W,Z)$ and $\cF_{\cZ}$,
$\cF_{\cW, \cZ}$ are the $\sigma$-algebras generated by the
projections from $\cS \times \cW \times \cZ$ to 
$\cZ$ and $\cW \times \cZ$, respectively. This shows 
that conditions \eqref{condindep1}--\eqref{condindep2}
only depend on the distribution of $(S,W,Z)$.

If $\cS$ and $\cW$ are Polish spaces equipped with their Borel $\sigma$-algebras,
$(S,W)$ has a regular conditional distribution $\rho(ds, dw \mid z)$ given $Z = z$
with respect to $\rho$ (see e.g. \citep[Theorem 3.4]{kal3rd}). In this case,
conditions \eqref{condindep1}--\eqref{condindep2} can equivalently be written as
\[
\rho(ds, dw \mid z) = \rho(ds \mid z) 
\rho(dw \mid z) \quad \mbox{for $Z_{\#} \bP$-almost all $z \in \cZ$.}
\]

\begin{definition}[$G$-compatible probability measures]
	\label{def:Gcomp}
	A measure $\mu \in \mathcal{P}(\cX)$ is said to be compatible with a sorted DAG 
	$G = (V,E)$ if any of the following three equivalent conditions hold. We denote the set of $G$-compatible measures in $\cP(\cX)$ by $\mathcal{P}_G(\cX)$.
	\begin{itemize}
		\item[(i)] There exists a random variable $X \sim \mu$ together with measurable functions
		$f_i : \cX_{\pa_i} \times \mathbb{R} \rightarrow  \cX_i$, $i = 1, \dots, n$, 
		and independent $\mathbb{R}$-valued random variables 
		$U_1, \dots, U_n$ such that
		\be \label{XfU}
		X_i = f_i(X_{\pa_i}, U_i) \quad \text{for all $i =1 , \dots, n$}.
		\ee
		\item[(ii)] 
		For every random variable $X \sim \mu$, one has
		\begin{equation} \label{Xindep}
		X_i \ind_{X_{\pa_i}} X_{1:i-1} \quad \mbox{for all $i = 2, \dots, n.$}
		\end{equation}
		\item[(iii)] 
		
		The measure $\mu$ can be disintegrated as
		\be \label{mudis}
		\mu(dx_1, \dots, dx_n) = \prod_{i=1}^n \mu \brak{dx_i \mid x_{\pa_i}}
		\ee
		for $\mu(dx_1 \mid x_{\pa_1}) = \mu(dx_1)$ and 
		regular conditional distributions $\mu(dx_i \mid x_{\pa_i})$, $i = 2, \dots, n$,
		called {\sl causal mechanisms} in causal modelling.

	\end{itemize}
\end{definition}

\begin{remark}
	The equivalence of conditions (i)--(iii) in Definition \ref{def:Gcomp} 
	follows from the definitions \eqref{condindep1}--\eqref{condindep2}
	of conditional independence together with some well-known properties
	of probability measures on Polish spaces. For the sake of completeness, 
	we give a short proof:
	
	(i) $\Leftrightarrow$ (ii): If (i) holds, there exists an ${\cal X}$-valued random variable
	$X = (X_1, \dots, X_n) \sim \mu$ 
	such that the $X_i$ are of the form $X_i = f_i(X_{\pa_i}, U_i)$ for measurable mappings 
	$f_i \colon \cX_{\pa_i} \times \R \to \cX_i$ and
	independent $\mathbb{R}$-valued random variables $U_i$, $i=1, \dots, n$,
	defined on a probability space $(\Omega, \cF, \bP)$. 
	It follows that $\bP[X_i \in A \mid X_{1:i-1}] = \bP[X_i \in A \mid X_{\pa_i}]$ 
	$\bP$-a.s. for all measurable subsets $A \subseteq \cX_i$, 
	which, by \eqref{condindep2}, shows that (ii) is satisfied.
	
	On the other hand, if (ii) holds, one obtains from \citep[Proposition 8.20]{kal3rd} that, 
	possibly after extending the underlying probability space, the $X_i$ have 
	representations of the form $X_i = f_i(X_{\pa_i}, U_i)$ $\bP$-a.s.
	for measurable functions $f_i \colon \cX_{\pa_i} \times \bR \to \cX_i$ 
	and $U(0,1)$-distributed random variables $U_i$ satisfying 
	$U_i \ind X_{1:i-1}$ for all $i = 2, \dots, n$.
	So $X$ is of the form \eqref{XfU}, except that the representation
	only holds $\bP$-a.s. and the random variables $U_1, \dots, U_n$ 
	are not necessarily independent. But this can easily be remedied by considering 
	independent $U(0,1)$-distributed random variables $\tilde{U}_1, \dots, \tilde{U}_n$
	on a new probability space and iteratively defining 
	$\tilde{X}_i = f_i(\tilde{X}_{\pa_i}, \tilde{U}_i)$, $i = 1, \dots, n$. Then 
	$\tilde{X}_1$ has the same distribution as $X_1$, and for all $i \ge 2$, 
        the conditional distribution of $\tilde{X}_i$ given $\tilde{X}_{1:i-1}$ is 
	equal to the conditional distribution of $X_i$ given $X_{1:i-1}$. So, it follows 
	by induction over $i$ that $\tilde{X} = (\tilde{X}_1, \dots, \tilde{X}_n)$ has the 
	same distribution as $X = (X_1, \dots, X_n)$ and satisfies condition (i).
	
	(ii) $\Leftrightarrow$ (iii): Since $\cX_1, \dots, \cX_n$ are Polish spaces, one obtains 
	from \citep[Theorem 3.4]{kal3rd}  that $\mu$ admits 
	an iterated disintegration of the form
	\[
	\mu(dx_1, \dots, dx_n) = \mu(dx_1) \mu \brak{dx_2 \mid x_1} 
	\cdots \mu \brak{dx_n \mid x_{1:n-1}}
	\]
for regular conditional distributions $\mu (dx_i \mid x_{1:i-1})$, $i=2, \dots, n$,	
from which it can be seen that \eqref{Xindep} is equivalent to \eqref{mudis}.
\end{remark}

The following result provides some elementary properties of the set $\cP_G(\cX)$.

\begin{proposition}
	Let $G = (V,E)$ be a sorted DAG. Then
	\begin{itemize}
		\item[{\rm (i)}] $\cP_G(\cX)$ is non-empty,
		\item[{\rm (ii)}] $\cP_G(\cX)$ is closed in total variation,
		\item[{\rm (iii)}] if ${\cal X}$ is a discrete\footnote{i.e. $\cX$ is a countable 
			discrete topological space} Polish space, 
		$\cP_G(\cX)$ is weakly closed,
		\item[{\rm (iv)}] but there exist Polish spaces $\mathcal{X} = \mathcal{X}_1 \times \dots \times \mathcal{X}_n$ and graphs $G$ such that $\cP_G(\cX)$ is not weakly closed.
	\end{itemize}
\end{proposition}

\begin{proof}
	According to Definition \ref{def:Gcomp}.(iii), the product measure 
	$\mu_1 \otimes \dots \otimes \mu_n$ belongs to $\cP_G(\cX)$
	for arbitrary marginal measures $\mu_i \in \cP(\cX_i)$, $i = 1, \dots, n$.
	This shows (i).
	
	Moreover, it follows from Lemma \ref{lemma:stable} in Appendix \ref{app:stability} below 
	that Definition \ref{def:Gcomp}.(ii) is stable under convergence in total variation, which proves (ii).
	
	(iii) follows from (ii) since in case $\cX$ is a discrete Polish space, weak convergence is equivalent to convergence in total variation.
	
	To show (iv), we consider the Markovian graph $G = (V,E)$ given by $V = \{1, 2, 3\}$
	and $G = \{(1,2), (2,3)\}$. Let $\cX_1 = \cX_2 = \cX_3 = [-1,1]$ and consider the
	product space $\cX = [-1,1]^3$ with the Borel $\sigma$-algebra. The measures
	$\mu^k = \frac{1}{2} \delta_{(1, 1/k, 1)} + \frac{1}{2} \delta_{(-1, -1/k, -1)}$, $k \in \bN$,
	belong to $\cP_G(\cX)$ for all $k \in \bN$ and weakly converge to
	the measure $\frac{1}{2} \delta_{(1, 0, 1)} + \frac{1}{2} \delta_{(-1, 0, -1)}$,
	which does not belong to $\cP_G(\cX)$. This proves (iv).
\end{proof}

The following is the main structural result of this paper. It gives different characterizations of 
$G$-causal couplings between two measures 
$\mu \in \cP_G(\cX)$ and $\nu \in \cP(\cY)$.
Note that for a sorted DAG $G$, one has $\Pi_{G}(\mu, \nu) \subseteq \Pi_{G_l}(\mu, \nu)$, where $G_l$ is the linear graph from Section \ref{ss:linear}, which leads to causal optimal transport between distributions of stochastic processes, for which related results have been shown by e.g.~\citep[Proposition 2.4]{backhoff2017causal}.

\begin{theorem} \label{thm:main}
	Let $G = (V,E)$ be a sorted DAG, and consider measures 
	$\mu \in \mathcal{P}_G(\cX)$, $\nu \in \mathcal{P}(\cY)$ and $\pi \in \Pi(\mu, \nu)$. Then, 
	the following are equivalent:
	\begin{itemize}
		\item[{\rm (i)}] $\pi \in \Pi_{G}(\mu, \nu)$,
		\item[{\rm (ii)}] for $(X, Y) \sim \pi$, one has
		\[ Y_i \ind_{X_{i}, X_{\pa_i}, Y_{\pa_i}} (X, Y_{1:i-1}) 
		\quad \mbox{for all } i = 1 ,\dots, n, \mbox{ where $Y_{1: 0}$ is the
		empty tuple,}
		\]
		\item[{\rm (iii)}] for $(X, Y) \sim \pi$, one has
		\[
		X_i \ind_{X_{\pa_i}} (X_{1:i-1}, Y_{1:i-1}) \quad \mbox{and} \quad
		Y_i \ind_{X_i, X_{\pa_i}, Y_{\pa_i}} (X_{1:i}, Y_{1:i-1}) 
		\quad \mbox{for all } i = 2,\dots, n,
		\]
		\item[{\rm (iv)}] $\pi$ is jointly $G$-compatible, that is,
		\[
		\pi(dx_1, dy_1, \dots, dx_n, dy_n)
		= \prod_{i=1}^n \pi(dx_i, dy_i \mid x_{\pa_i}, y_{\pa_i}),
		\] for regular conditional distributions $\pi(dx_i, dy_i \mid x_{\pa_i}, y_{\pa_i})$, and
		\[
		\pi(d x_i \mid x_{\pa_i}, y_{\pa_i}) = \mu(d x_i \mid x_{\pa_i})
		\quad  \mbox{for all $i = 1, \dots, n$ and 
			$\pi$-almost all $(x,y) \in \cX\times\cY$,}\]
		where $\mu(d x_i \mid x_{\pa_i})$ are regular conditional versions of $\mu$.
	\end{itemize}
\end{theorem}

Before turning to the proof of Theorem \ref{thm:main}, we note that 
it implies the following alternative characterizations of $G$-bicausal couplings
between two measures $\mu \in \cP_G(\cX)$ and $\nu \in \cP_G(\cY)$. Particularly the characterization in part (iii) of Corollary \ref{cor:main} is worth emphasizing, as this characterization is very natural and hence serves as validation of Definition \ref{def:causal}. The characterization states that, just like the marginals, a $G$-bicausal coupling also has to be $G$-compatible. And on top of that, \emph{the causal mechanisms of the marginals are coupled by the causal mechanisms of the coupling}. In this way, $G$-bicausal couplings can intuitively be understood as fine-grained couplings, which put the essential building blocks of causal models, the causal mechanisms, into relation.

\begin{corollary} \label{cor:main}
Let $G = (V,E)$ be a sorted DAG, $\mu \in \mathcal{P}_G(\cX)$, 
	$\nu \in \mathcal{P}_G(\cY)$ and $\pi \in \Pi(\mu, \nu)$. 
	Then, the following are equivalent:
	\begin{itemize}
		\item[{\rm (i)}] 
		$\pi \in \Pi^{\rm bc}_G(\mu, \nu)$,
		\item[{\rm (ii)}] 
		for $(X, Y) \sim \pi$, one has
		\[ 
		(X_i, Y_i) \ind_{X_{\pa_i}, Y_{\pa_i}} (X_{1:i-1}, Y_{1:i-1}),
		\quad
		X_i \ind_{X_{\pa_i}} Y_{\pa_i} \quad \mbox{and} 
		\quad  Y_i \ind_{Y_{\pa_i}} X_{\pa_i}
		\quad \mbox{for all } i =2,\dots, n,
		\]
		\item[{\rm (iii)}] $\pi$ is jointly $G$-compatible, that is,
		\[
		\pi(dx_1, dy_1, \dots, dx_n, dy_n) = \prod_{i=1}^n \pi(dx_i, dy_i \mid 
		x_{\pa_i}, y_{\pa_i})
		\]
		for regular conditional distributions $\pi(dx_i, dy_i \mid x_{\pa_i}, y_{\pa_i})$, and
		\[
		\pi(dx_i, dy_i \mid x_{\pa_i}, y_{\pa_i}) \in \Pi(\mu(dx_i \mid x_{\pa_i}), \nu(dy_i \mid y_{\pa_i}))
		\]
		for all $i = 1, \dots, n$ and $\pi$-almost all $(x, y) \in \cX \times \cY$,
		where $\mu(dx_i \mid x_{\pa_i})$ and $\nu(dy_i \mid y_{\pa_i})$
		are regular conditional versions of $\mu$ and $\nu$, respectively.
	\end{itemize}
\end{corollary}

\noindent
{\textbf{Proof of Theorem \ref{thm:main}.}}
(i) $\Rightarrow$ (ii): If (i) holds, there exists $(X,Y) \sim \pi$ such that 
the $Y_i$ are of the form $Y_i = g_i(X_i, X_{\pa_i}, Y_{\pa_i}, U_i)$ for measurable functions 
$g_i \colon \cX_i \times \cX_{\pa_i} \times \cY_{\rm pa_i} \times \bR \to \cY_i$ 
and $\R$-valued random variables $U_i$, $i = 1, \dots, n$ defined on the same 
probability space as $(X,Y)$ such that $X,U_1, \dots, U_n$ are independent. It follows that
$U_i \ind (X, Y_{1:i-1})$, and therefore, 
$\bP[Y_i \in A \mid X, Y_{1:i-1}] = \bP[Y_i \in A \mid X_i, X_{\pa_i}, Y_{\pa_i}]$ $\bP$-a.s.
for every $i$ and all measurable subsets $A \subseteq \cY_i$, which by 
\eqref{condindep2}, shows that (ii) is satisfied.

(ii) $\Rightarrow$ (i):
If $(X,Y)$ fulfils (ii), one obtains 
from \citep[Proposition 8.20]{kal3rd} that, possibly after extending the underlying 
probability space, the $Y_i$ can be represented as
$Y_i = g_i(X_i, X_{\pa_i}, Y_{\pa_i}, U_i)$ $\bP$-a.s.
for measurable functions $g_i \colon \cX_i \times \cX_{\pa_i} \times \cY_{\pa_i} \times \R \to \cY_i$
and $U(0,1)$-distributed random variables $U_i$ satisfying
$U_i \ind (X, Y_{1:i-1})$ for all $i = 1, \dots, n$. So $(X,Y)$ is of the form 
\eqref{condcausal}, except that the representation only holds $\bP$-a.s.
and the random variables $U_1, \dots, U_n$ are not 
necessarily independent. However, by extending the probability space further
if necessary, one can assume that there exist $U(0,1)$-distributed random 
variables $\tilde{U}_1, \dots, \tilde{U}_n$ 
such that $X, \tilde{U}_1, \dots, \tilde{U}_n$ are independent. If one iteratively
defines $\tilde{Y}_i = g_i(X_i, X_{\pa_i}, \tilde{Y}_{\pa_i}, \tilde{U}_i)$, $i = 1, \dots, n$,
then the conditional distribution of $\tilde{Y}_i$ given $(X_i, X_{\pa_i}, \tilde{Y}_{\pa_i})$
is for all $i$ equal to the conditional distribution of $Y_i$ given $(X_i, X_{\pa_i}, Y_{\pa_i})$.
It follows by induction over $i$ that $(X,\tilde{Y})$ has the same distribution 
as $(X,Y)$, and it satisfies all requirements of Definition \ref{def:causal}. So,
(i) holds.

(ii) $\Rightarrow$ (iii): The second condition of (iii) is immediate from (ii).
In addition, one obtains from (ii) that $Y_j \ind_{X_{j}, X_{\pa_j}, Y_{\pa_j}} X_i$ 
for all $1 \le i, j \le n$. In particular, $X_i \ind_{X_{1:i-1}, Y_{1:j-1}} Y_j$ 
for all $1 \le j < i \le n$, which, by the chain rule \citep[Theorem 8.12]{kal3rd} yields 
$X_i \ind_{X_{1:i-1}} Y_{1:i-1}$. This together with 
$X_i \ind_{X_{\pa_i}} X_{1:i-1}$, which holds since $\mu \in \cP_G(\mu)$,
gives $X_i \ind_{X_{\pa_i}} (X_{1:i-1}, Y_{1:i-1})$ for all $i = 2, \dots, n$, showing that also the
first condition of (iii) holds.

(iii) $\Rightarrow$ (ii):
The first condition of (iii) implies
$X_j \ind_{X_{1:j-1}} Y_{1:i}$, and therefore,
$Y_i \ind_{X_{1:j-1}, Y_{1:i-1}} X_j$ for all $1 \le i < j \le n$. 
This, together with the second condition of (iii) and the 
chain rule \citep[Theorem 8.12]{kal3rd}, yields (ii). 

(iii) $\Leftrightarrow$ (iv): Since $\cX_i$, $\cY_i$, $i = 1, \dots, n$, are Polish 
spaces, it follows from \citep[Theorem 3.4]{kal3rd} that $\pi$ can be disintegrated as
\[
\pi(dx_1, dy_1, \dots, dx_n, dy_n) 
= \prod_{i=1}^n \pi \brak{dx_i \mid x_{1:i-1}, y_{1:i-1}}
\pi \brak{dy_i \mid x_{1:i} \, , y_{1:i-1}},
\]
where $x_{1:0}$ and $y_{1:0}$ are empty tuples.
Since $\pi \in \Pi(\mu, \nu)$, one has
\[
X_i \ind_{X_{\pa_i}} (X_{1:i-1}, Y_{1:i-1}) \quad \Leftrightarrow \quad
\pi(d x_i \mid x_{1:i-1}, y_{1:i-1}) = \mu(d x_i \mid x_{\pa_i})
\quad  \mbox{$\pi$-almost surely}
\]
and
\[
Y_i \ind_{X_i, X_{\pa_i}, Y_{\pa_i}} (X_{1:i}, Y_{1:i-1}) \quad \Leftrightarrow \quad
\pi \brak{dy_i \mid x_{1:i} \, , y_{1:i-1}} =
\pi \brak{dy_i \mid x_i, x_{\pa_i} , y_{\pa_i}}
\quad  \mbox{$\pi$-almost surely,}
\]
which shows that (iii) is equivalent to 
\[
\pi(dx_1, dy_1, \dots, dx_n, dy_n) 
= \prod_{i=1}^n \mu \brak{dx_i \mid x_{\pa_i}}
\pi \brak{dy_i \mid x_i, x_{\pa_i} , y_{\pa_i}},
\]
which, in turn, is equivalent to (iv). \qed

\bigskip \noindent
{\textbf{Proof of Corollary \ref{cor:main}.}} 
(i) $\Leftrightarrow$ (iii) follows directly from the equivalence of (i) and (iv) 
of Theorem \ref{thm:main}.

(ii) $\Leftrightarrow$ (iii): As in the proof of Theorem \ref{thm:main}, 
we obtain from \citep[Theorem 3.4]{kal3rd} that $\pi$ can be disintegrated as
\[
\pi(dx_1, dy_1, \dots, dx_n, dy_n) 
= \prod_{i=1}^n \pi \brak{dx_i, dy_i \mid x_{1:i-1}, y_{1:i-1}},
\]
where $x_{1:0}$ and $y_{1:0}$ are empty tuples, 
from which it can be seen that (ii) is equivalent to (iii). \qed

\bigskip
In the following result we derive some elementary properties of the sets
$\Pi_{G}(\mu, \nu)$ and $\Pi^{\rm bc}_{G}(\mu, \nu)$.

\begin{proposition} \label{prop:basic_prop}
	Let $G = (V,E)$ be a sorted DAG and
	$\mu \in \mathcal{P}_G(\cX)$, $\nu \in \mathcal{P}_G(\cY)$. Then
	\begin{itemize}
		\item[{\rm (i)}] $\Pi_{G}(\mu, \nu)$ and $\Pi^{\rm bc}_{G}(\mu, \nu)$ are both non-empty,
		\item[{\rm (ii)}] $\Pi_{G}(\mu, \nu)$ and $\Pi^{\rm bc}_{G}(\mu, \nu)$ are closed in
		total variation,
		\item[{\rm (iii)}] 
		if $\mu$ and $\nu$ are finitely supported,
		$\Pi_{G}(\mu, \nu)$ and $\Pi^{\rm bc}_{G}(\mu, \nu)$ are weakly closed,
		\item[{\rm (iv)}] but in general, there exist $\cX$, $\cY$ and $G$ such that $\Pi_{G}(\mu, \nu)$ and $\Pi^{\rm bc}_{G}(\mu, \nu)$ are not weakly closed.
	\end{itemize}
\end{proposition}

\begin{proof}
	Since $\mu$ and $\nu$ are both compatible with $G$, they 
	admit a decomposition of the form \eqref{mudis}. So, it can be seen 
	from Corollary \ref{cor:main} that the product measure $\mu \otimes \nu$ 
	belongs to $\Pi^{\rm bc}_{G}(\mu, \nu)$, which is contained in $\Pi_{G}(\mu, \nu)$.
	This shows (i).
	
	By Lemma \ref{lemma:stable} in Appendix \ref{app:stability}, conditions (ii) and (iii) 
	of Theorem \ref{thm:main} are stable under convergence in total variation, which
	implies (ii).
	
	Regarding (iii), if $\mu$ and $\nu$ have finite supports $S_{\mu} \subseteq \cX$ and 
	$S_{\nu} \subseteq \cY$, respectively, every measure $\pi \in \Pi(\mu, \nu)$ is 
	supported by the finite subset $S_{\mu} \times S_{\nu} \subseteq \cX \times \cY$. Since
	weak convergence among such measures is equivalent to convergence in 
	total variation, (iii) follows from (ii).
	
	To show (iv), we consider a Markovian graph $G = (V,E)$ with $V = \{1, 2, 3\}$ and 
	$E = \{(1,2), (2,3) \}$. Consider the measures
	$\mu = \nu = \theta \otimes \eta \otimes \theta$ on $[0,1]^3$, where 
	$\theta = \frac{1}{2}(\delta_{0} + \delta_1)$ and $\eta$ is the uniform 
	distribution on $[0,1]$. We show that $\Pi^{\rm bc}_G(\mu, \nu)$ is not weakly 
	closed. Then $\Pi_G(\mu, \nu)$ cannot be weakly closed either. Otherwise, 
	$\Pi_G(\nu, \mu)$ and, as a consequence, $\Pi^{\rm bc}_G(\mu, \nu)$ 
	would also have to be weakly closed. To prove that $\Pi^{\rm bc}_G(\mu, \nu)$ 
	is not weakly closed, we denote by $\gamma \in \Pi(\eta, \eta)$ the uniform 
	distribution on the diagonal $D = \{(x,x) \in [0,1]^2 : x \in [0,1] \}$ and by 
	$\gamma^{1,k}, \gamma^{2,k} \in \Pi(\eta, \eta)$, $k \ge 3$, the uniform distributions 
	on the shifted diagonals $D^{1, k}, D^{2, k}$ given by
	\begin{align*}
	D^{1, k} &= \left\{(x, x+ 1/k): x \in [0, 1- 1/k]\right\} 
	\cup \left\{(x, x+ 1/k -1) : x \in (1- 1/k, 1]\right\}, \\
	D^{2, k} &= \left\{(x, x + 1- 1/k) : x \in [0, 1/k]\right\} \cup 
	\left\{[x, x- 1/k) : x \in [1/k, 1]\right\}.
	\end{align*}
	Clearly, $\gamma^{1,k}$ and $\gamma^{2,k}$ both weakly converge to $\gamma$
	for $k \to \infty$. Now, let us define the measures $\pi^k$, $k \ge 3$, on $([0,1]^2)^3$
	through 
	\[
	\pi^k(dz_1, dz_2, dz_3) = \pi_1^k(dz_1) \pi_{1, 2}^k(z_1, dz_2) \pi_{2, 3}^k(z_2, dz_3)
	\] 
	for
	\[
	\pi_1^k = \frac{1}{2}(\delta_{(0, 0)} + \delta_{(1, 1)}), \quad
	\pi_{1, 2}^k(0, 0) = \gamma^{1, k}, \quad
	\pi_{1, 2}^k(1, 1) = \gamma^{2, k}
	\]
	and
	\[
	\pi_{2, 3}^k(z_2) = 
	\begin{cases}
	\frac{1}{2} (\delta_{(0, 0)} + \delta_{(1, 1)}) \text{ if } z_2 \in D^{1, k}\\
	\frac{1}{2} (\delta_{(0, 1)} + \delta_{(1, 0)}) \text{ if } z_2 \in D^{2, k}.
	\end{cases}
	\]
	Then, $\pi^k \in \Pi^{\rm bc}_G(\mu, \nu)$ for all $k \ge 3$, and for $k \to \infty$, $\pi^k$ weakly 
	converges to $\pi$ given by
	\[
	\pi(dz_1, dz_2, dz_3) = \pi_1(dz_1) \pi_{1, 2}(dz_2) \pi_{2, 3}(z_1, dz_3)
	\] 
	for
	\[
	\pi_1 = \frac{1}{2}(\delta_{(0, 0)} + \delta_{(1, 1)}), \quad
	\pi_{1, 2} = \gamma
	\]
	and
	\[
	\pi_{2, 3}(z_1) = 
	\begin{cases}
	\frac{1}{2} (\delta_{(0, 0)} + \delta_{(1, 1)}) \text{ if } z_1 = (0,0)\\
	\frac{1}{2} (\delta_{(0, 1)} + \delta_{(1, 0)}) \text{ if } z_1 = (1,1).
	\end{cases}
	\]
	So $\pi$ is no longer Markovian, and therefore, $\pi \not\in \Pi^{\rm bc}_G(\mu, \nu)$, 
	which proves (iv).
\end{proof}

\section{$G$-causal Wasserstein distances}
\label{sec:Wasserstein}

In this section, we study Wasserstein distances arising from $G$-bicausal couplings.
We show that they define semimetrics on appropriate spaces of $G$-compatible probability
measures but in general do not satisfy the triangle inequality. In Section \ref{subsec:ate} 
we prove that under suitable assumptions,
average treatment effects are continuous with respect to $G$-causal Wasserstein 
distance, and in Section \ref{subsec:closeclose} we show that small perturbations of structural 
causal models lead to small deviations in $G$-causal Wasserstein distance. 
In Subsection \ref{subsec:geodesic} we investigate the interpolation of two $G$-compatible 
measures such that the causal structure given by $G$ is respected.
In the whole section, we let $\cX = \cY$ and denote by $d_{\cX}$ a metric 
$d_{\cX} \colon \cX \times \cX \to \bR_+$ generating\footnote{we assume
	$d_{\cX}$ generates the topology, but $\cX$ is not necessarily complete with respect to 
	$d_{\cX}$} the given Polish topology on $\cX$.

\begin{remark}[Standard Wasserstein distances]
	For a $p \in [1, \infty)$, we denote by $\cP_p(\cX)$ the set of all 
	measures $\mu \in \cP(\cX)$ satisfying 
	\be \label{x0}
	\int_{\cX} d_{\cX}(x, x_0)^p \mu(dx) < \infty \quad \mbox{for some } x_0 \in \cX.
	\ee
	Note that \eqref{x0} implies 
	\[
	\int_{\cX} d_{\cX}(x, x'_0)^p \mu(dx)
	\le 2^{p-1} \brak{ \int_{\cX} d_{\cX}(x,x_0)^p \mu(dx) + 
		d_{\cX}(x_0, x'_0)^p} < \infty \quad \mbox{for every } x'_0 \in \cX.
	\]
	So, if \eqref{x0} holds for one, it holds for all $x_0 \in \cX$. 
	The standard $p$-Wasserstein distance $W_p(\mu, \nu)$ between two 
	measures $\mu$ and $\nu$ in $\cP_p(\cX)$ is given by 
	\[
	W_p^p(\mu, \nu) := \inf_{\pi \in \Pi(\mu, \nu)} 
	\int_{\cX \times \cX} d_{\cX}(x, y)^p \,\pi(dx, dy).
	\]
	It is well known (see e.g. \citep{villani2009optimal}) that the infimum is attained and
	$W_p$ defines a metric on $\cP_p(\cX)$. Moreover, 
	$\Pi(\mu, \nu)$ is a weakly compact subset of $\cP(\cX \times \cX)$.
\end{remark}
In the following definition we introduce Wasserstein distances based on $G$-bicausal couplings. For causal Wasserstein distances between distributions of stochastic processes, see \citep{backhoff2020all} and the references therein.
\begin{definition}[$G$-causal Wasserstein distances]\label{def:gwasserstein}
	Let $G = (V,E)$ be a sorted DAG and $p \in [1, \infty)$. 
	We denote $\cP_{G,p}(\cX) = \cP_G(\cX) \cap \cP_p(\cX)$
	and define the $p$-th order $G$-causal Wasserstein distance 
	$W_{G, p}(\mu, \nu)$ between two measures $\mu$ and 
	$\nu$ in $\mathcal{P}_{G,p}(\cX)$ by
	\be \label{WGp}
	W^p_{G, p}(\mu, \nu) 
	:= \inf_{\pi \in \Pi^{\rm bc}_G(\mu, \nu)} \int_{\cX \times \cX} d_{\cX}(x, y)^p \,\pi(dx, dy).
	\ee
\end{definition}

The following result provides some elementary properties of 
$G$-causal Wasserstein distances.

\begin{proposition} \label{prop:wassersteinprop}
	Let $G = (V,E)$ be a sorted DAG and $p \in [1, \infty)$. Then
	\begin{itemize}
		\item[{\rm (i)}] for another directed graph $G' = (V, E')$ with $E' \subseteq E$, one has 
		$W_{G, p} \leq W_{G', p}$, and therefore,
		\[
		\{\nu \in \mathcal{P}(\cX) : W_{G', p}(\mu, \nu) \leq r \} 
		\subseteq \{\nu \in \mathcal{P}(\cX) : W_{G, p}(\mu, \nu) \leq r \}
		\]
		for all $\mu \in \mathcal{P}(\cX)$ and $r > 0$, 
		\item[{\rm (ii)}] $W_{G,p}$ is a semimetric\footnote{that is,
			it has all properties of a metric except the triangle inequality} on $\mathcal{P}_{G,p}(\cX)$
		and
		\item[{\rm (iii)}] if $\mu, \nu \in \cP_{G,p}(\cX)$ are finitely supported, the infimum 
		in \eqref{WGp} is attained.
	\end{itemize} 
\end{proposition}
\begin{proof}
	(i): If $E' \subseteq E$, one has $\pa'_i \subseteq \pa_i$, $i = 1, \dots, n$, for the 
	corresponding parent sets. Therefore, it can be seen from Definition \ref{def:causal} 
	that $\Pi^{\rm bc}_{G'}(\mu, \nu) \subseteq \Pi^{\rm bc}_G(\mu, \nu)$ for all $\mu, \nu \in \cP(\cX)$,
	which implies (i).
	
	(ii) It is clear that $W_{G,p}$ is symmetric and non-negative. 
	Moreover, it follows from
	Proposition \ref{prop:basic_prop}.(i) that for 
	all $\mu, \nu \in \cP_{G,p}$, there exists a coupling $\pi \in \Pi^{\rm bc}_G(\mu, \nu)$.
	Therefore, one has for any $x_0 \in \cX$, 
	\beas
	W^p_G(\mu, \nu) &\le& \int_{\cX \times \cX} 
	(d_{\cX}(x,y))^p \pi(dx, dy) \le \int_{\cX \times \cX} 
	(d_{\cX}(x,x_0) + d_{\cX}(y,x_0) )^p \pi(dx, dy) \\
	&\le& 2^{p-1} \brak{ \int_{\cX} 
		(d_{\cX}(x,x_0))^p \mu(dx) + \int_{\cX} 
		(d_{\cX}(x, x_0))^p \nu(dx)} < \infty.
	\eeas
	Next, note that for $X \sim \mu \in \cP_{G,p}(\cX)$, the distribution of $(X,X)$ belongs to 
	$\Pi^{\rm bc}_{G}(\mu, \mu)$, from which one obtains $W_{G, p}(\mu, \mu) = 0$.
	Finally, since $\Pi^{\rm bc}_G(\mu, \nu) \subseteq \Pi(\mu, \nu)$, one has 
	\[
	W_{G, p}(\mu, \nu) \geq W_p(\mu, \nu) > 0 \quad \mbox{if } \mu \neq \nu,
	\]
	where the last inequality holds since the standard $p$-Wasserstein distance 
	$W_p$ is a metric.
	
	(iii) If $\mu, \nu \in \cP_{G,p}(\cX)$ are finitely supported, choose
	a sequence $(\pi^k)_{k \ge 1}$ in $\Pi^{\rm bc}_G(\mu, \nu)$ such that 
	\be \label{limk}
	\int_{\cX} d_{\cX}(x,y)^p d\pi^k(dx, dy) \to W_{G,p}^p(\mu, \nu) \quad
	\mbox{for } k \to \infty.
	\ee
	Since $\Pi(\mu, \nu)$ is a weakly compact subset of
	$\cP(\cX \times \cX)$, there exists a subsequence, again denoted $(\pi^k)_{k \ge 1}$, 
	which weakly converges to a measure $\pi \in \Pi(\mu, \nu)$.
	By Proposition \ref{prop:basic_prop}.(iii), $\pi$ belongs to $\Pi^{\rm bc}_G(\mu, \nu)$, 
	and, by \eqref{limk}, it minimizes \eqref{WGp}.
\end{proof}

We show in Appendix \ref{sec:counter} below that in general, 
$W_{G,p}$ does not satisfy the triangle inequality.
\subsection{Continuity of average treatment effects with respect to 
$G$-causal Wasserstein distance}
\label{subsec:ate}

The average treatment effect is a measure of the causal effect of a treatment or intervention (cf.~\citep{rosenbaum1983central}). 
In the following we show that under suitable assumptions,
it is Lipschitz continuous with respect to $W_{G,1}$ 
under a change of the underlying probability model while it is not continuous 
with respect to the standard Wasserstein distance. I.e., we are interested in the mapping from probability measures to average treatment effect. An important implication of this continuity result is that 
one can obtain bounds on the error resulting from estimating average treatment
effects from data by controlling the $G$-causal Wasserstein estimation error. For recent 
studies on how average treatment effects can be estimated from data, see e.g.
\citep{savje2021average, huang2022robust}. 

Let us assume there exist two indices $j,k \in V = \{1, \dots, n\}$ such that $j < k$, 
$\cX_j = \{0, 1\}$ and $\cX_k$ is a compact subset of $\bR$. 
$X_j$ indicates whether a treatment is applied 
or not, and $X_k$ describes a resulting outcome. Suppose the metric $d_{\cX}$ on 
$\cX = \cX_1 \times \dots \times \cX_n$ is 
given by $d_{\cX}(x,y) := \sum_{i=1}^n d_{\cX_i}(x_i, y_i)$, where 
$d_{\cX_i}(x_i, y_i) = |x_i - y_i|$ for $i = j, k$.
Using Pearl's do-notation (cf. the back-door adjustment 
\citep[Theorem 3.3.2]{pearl2009causality} applied with $\pa_j$),
the average treatment effect under a model $\mu \in \cP_G(\cX)$ is given by
\be \label{ate}
\begin{aligned}
	\psi^\mu &:= \int_{\cX_k} x_k \,\mu(dx_k \mid \Do(x_j = 1)) - \int_{\cX_k}
	x_k \, \mu(dx_k \mid \Do(x_j = 0)) \\
	&= \int_{\cX_{\pa_i}} \int_{\cX_k} x_k \,\mu(dx_k \mid x_j = 1, x_{\pa_j}) \, \mu(dx_{\pa_j}) 
	-\int_{\cX_{\pa_i}} \int_{\cX_k} x_k \,\mu(dx_k \mid x_j = 0, x_{\pa_j}) \, \mu(dx_{\pa_j}),
\end{aligned}
\ee
where for the purposes of this paper, the second line 
serves as definition of the do-notation in the first line.
In the following, we fix a constant $\delta > 0$ and consider the set 
$\mathcal{P}_{G}^{\delta, j}(\cX)$ of models $\mu \in \mathcal{P}_G(\cX)$ 
such that the propensity score $\mu(x_j = 1 \mid x_{\pa_j})$ satisfies 
\[
\delta \le \mu(x_j = 1 \mid x_{\pa_j}) \le 1 - \delta 
\quad \mbox{for $\mu$-almost all $x_{\pa_j}$}.
\]
Then the following holds.

\begin{proposition}\label{prop:ate}
	Let $G= (V,E)$ be a sorted DAG, $\delta > 0$ and $j < k$, $\cX_j, \cX_k$, $d_{\cX}$ as above.
	Then there exists a constant $C \ge 0$ such that 
	\[
	\left| \psi^\mu - \psi^\nu \right| \leq C \, W_{G, 1}(\mu, \nu)
	\]
	for all $\mu, \nu \in \mathcal{P}_G^{\delta, j}(\cX)$.	
\end{proposition}

\begin{proof}
	Note that 
	\begin{align*}
	&\int_{\cX_k} x_k \,\mu(dx_k \mid \Do(x_j = 1)) \\
	&= \int_{\cX_{\pa_j}} \int_{\cX_k}  x_k  \, \mu(dx_k \mid x_j = 1, 
	x_{\pa_j}) \frac{\mu(x_j = 1 \mid x_{\pa_j})}{\mu(x_j = 1 \mid x_{\pa_j})} \, 
	\mu(dx_{\pa_j}) \\
	&= \int_{\cX_{\pa_j}} \int_{\cX_j} \int_{\cX_k} 
	\frac{x_k 1_{\crl{x_j = 1}} }{\mu(x_j = 1 \mid x_{\pa_j})}   \,
	\mu(dx_k \mid x_j, x_{\pa_j}) \mu(dx_j \mid x_{\pa_j}) \, \mu(dx_{\pa_j}) \\
	&= \int_{\cX_{\pa_j}} \int_{\cX_j} \int_{\cX_k} \frac{x_k 1_{\crl{x_j = 1}}}{\mu(x_j = 1 
		\mid x_{\pa_j})} \,\mu(dx_k, dx_j, dx_{\pa_j}) \\
	&= \int_{\cX}  \frac{x_k 1_{\crl{x_j = 1}}}{\mu(x_j = 1 \mid x_{\pa_j})} \,\mu(dx)\\
	&= \int_{\cX} f g h \, \mu(dx),
	\end{align*}
	where $f(x) = 1_{\{x_j = 1\}} = x_j$, $g(x) = x_k$ and  
	$h(x) = \mu(x_j = 1 \mid x_{\pa_j})^{-1}$. 
	Moreover, one has
	\beas
	&& f(x)g(x)h(x) - f(y)g(y)h(y)\\
	&=& [f(x) - f(y)]\, g(x)h(x) + f(y) [g(x) - g(y)] h(x)) + f(y) g(y) [h(x) - h(y)].
	\eeas
	Since by assumption, $\cX_k$ is compact and $\mu(x_j = 1\mid x_{\pa_j}) \ge \delta$,
	there exists a constant $K \ge 0$ such that 
	\[
	|g(x) h(x)| \le K, \quad |f(y)h(x)| \le K \quad \mbox{and} \quad 
	|f(y) g(y)| \le K \quad \mbox{for all} \quad x,y \in \cX.
	\]
	Let $\varepsilon > 0$ and choose a coupling $\pi \in \Pi^{\rm bc}_G(\mu, \nu)$ such that
	\[
	\int_{\cX^2} d_{\cX}(x,y) \pi(dx, dy) \le W_{G,1}(\mu, \nu) + \varepsilon.
	\]
	Then, 
	\beas
	&&\left| \int_{\cX_k} x_k \,\mu(dx_k \mid \Do(x_j = 1)) 
	- \int_{\cX_k} x_k \,\nu(dx_k \mid \Do(x_j = 1)) \right| 
	= \left| \int_{\cX} fgh \,d\mu - \int_{\cX} fgh \,d\nu \right| \\
	&=& \left| \int \edg{f(x)g(x)h(x) - f(y)g(y) h(y)} \,\pi(dx, dy)\right| \\
	&\leq & K \int_{\cX} \brak{|f(x) - f(y)| + |g(x) - g(y)| 
		+ |h(x) - h(y)|} \,\pi(dx, dy)
	\eeas
	It is clear that $\int_{\cX} \brak{|f(x) - f(y)| + |g(x) - g(y)|} \, \pi(dx, dy) \leq W_{G, 1}(\mu, \nu) 
	+ \varepsilon$. Moreover, by the lower bound on the propensity score and since
	\[
	\left| \frac{1}{a} - \frac{1}{b}\right| \leq \frac{1}{\delta^2} \, |a-b| \text{ for all } a, b \ge \delta,
	\]
	we obtain from Corollary \ref{cor:main} that
	\beas
	&& \delta^2 \int_{\cX^2} |h(x) - h(y)| \,\pi(dx, dy)
	\leq \int_{\cX^2} \left| \mu(x_j = 1 \mid x_{\pa_j}) - 
	\nu(y_j = 1 \mid y_{\pa_j}) \right| \,\pi(dx, dy) \\
	&=& \int_{\cX^2} \left| \pi(x_j = 1 \mid x_{\pa_j}, y_{\pa_j}) - \pi(y_j = 1 \mid x_{\pa_j}, 
	y_{\pa_j}) \right| \pi(dx, dy) \\
	&=& \int_{\cX^2_{\pa_j}} \left| \int_{\cX_j^2} (x_j - y_j) \pi(dx_j, dy_j \mid y_{\pa_j}, y_{\pa_j}) \right| 
	\pi(dx_{\pa_j}, dy_{\pa_j})\\
	&\leq& \int_{\cX^2} |x_j - y_j| \, \pi(dx, dy) \leq W_{G, 1}(\mu, \nu) + \varepsilon.
	\eeas
	Since $\varepsilon > 0$ was arbitrary, this shows that
	\be \label{Do1}
	\left| \int_{\cX_k} x_k \,\mu(dx_k \mid \Do(x_j = 1)) 
	- \int_{\cX_k} x_k \,\nu(dx_k \mid \Do(x_j = 1)) \right| 
	\le K \brak{1 + \frac{1}{\delta^2}} W_{G,1}(\mu, \nu).
	\ee
	Analogously, one deduces from the upper bound on the propensity score that
	\be \label{Do0}
	\left| \int_{\cX_k} x_k \,\mu(dx_k \mid \Do(x_j = 0)) 
	- \int_{\cX_k} x_k \,\nu(dx_k \mid \Do(x_j = 0)) \right| 
	\le K \brak{1 + \frac{1}{\delta^2}} W_{G,1}(\mu, \nu).
	\ee
	Now, the proposition follows from a combination of \eqref{Do1} and \eqref{Do0}.
\end{proof}

\begin{remark}
	It is easy to see that Proposition \ref{prop:ate} does not hold with the 
	standard $1$-Wasserstein distance $W_1$ instead of $W_{G,1}$ since 
	the topology generated by $W_1$ is not fine enough. Indeed, measures that are 
	close with respect to $W_1$ may have completely different 
	transition probabilities $\mu(dx_k \mid x_j = 1, x_{\pa_j})$, which play 
	a crucial role in the definition of the average treatment effect $\psi^\mu$; 
	see \eqref{ate}. 
	
	To construct a simple counterexample, we consider the temporal graph $G$ with three nodes (i.e., $E= \{(1, 2), (1, 3), (2, 3)\}$), where the second variable is the treatment variable and the third the output variable. For $p, H \in (0, 1)$, $q = (1-p), L = (1-H)$ and $\varepsilon > 0$, let $\mu^{\varepsilon}$ be given by
	\begin{align*}
	\mu^{\varepsilon} =  & \; pH^2 \delta_{(\varepsilon, 1, 1)} + pHL \delta_{(\varepsilon, 1, 0)} + pLH \delta_{(\varepsilon, 0, 1)} + pL^2 \delta_{(\varepsilon, 0, 0)} \\
	&+ qH^2 \delta_{(0, 0, 0)} + qHL \delta_{(0, 0, 1)} + qLH \delta_{(0, 1, 0)} + qL^2 \delta_{(0, 1, 1)}
	\end{align*} 
	Under $\mu^\varepsilon$, treatment and outcome variable are independent given the first variable. In other words, the treatment has no direct effect on the outcome and one readily finds $\psi^{\mu^{\varepsilon}} = 0$. On the other hand, consider the weak limit of $\mu^\varepsilon$ for $\varepsilon \rightarrow 0$, which is given by
	\[
	\mu^0 = (pH^2 + qL^2) \delta_{(0, 1, 1)} + (pL^2 + qH^2) \delta_{(0, 0, 0)} + LH \delta_{(0, 1, 0)} + LH \delta_{(0, 0, 1)}.
	\]
	For instance, setting $p=q=\frac{1}{2}$ and $H = 0.9, L=0.1$, we see that now there is a strong positive dependence between treatment and outcome variable, even conditional on the first variable (since conditioning on the first variable is now irrelevant). Indeed, with these values, one finds that $\psi^{\mu^0} = 0.64$. 
	
	The above example may intuitively be framed as modeling the question whether "drinking wine causes health benefits", where the first variable is the common cause "being wealthy or not", but the common cause is no longer included in the information structure for $\varepsilon = 0$.
\end{remark}

\subsection{Perturbation of structural causal models}
\label{subsec:closeclose}

We know from Proposition \ref{prop:wassersteinprop}.(i) that the topology 
generated by $W_{G,p}$ is the finer the fewer edges the graph $G$ has.
In Section \ref{subsec:ate} we saw that average treatment effects are continuous 
with respect to $W_{G,1}$ but not with respect to the standard $1$-Wasserstein distance $W_1$, 
which corresponds to a fully connected graph. In this section we show that 
the topology generated by $W_{G,1}$ is not too fine for practical purposes by 
proving that small perturbations of structural causal modes lead to distributions that 
are close with respect to $W_{G,1}$.

\begin{proposition}
	Let $G = (V,E)$ be a sorted DAG and consider for each $i=1, \dots, n$,
	a metric $d_{\cX_i}$ on $\cX_i$ generating the topology on $\cX_i$.
	Define the metric $d_{\cX}$ on 
	$\cX = \cX_1 \times \dots \times \cX_n$ by 
	$d_{\cX}(x, y) = \sum_{i=1}^n d_{\cX_i}(x_i, y_i)$ and consider
	$\cX$-valued random variables 
	$X \sim \mu \in \mathcal{P}_G(\cX)$ and $Y \sim \nu \in \mathcal{P}_G(\cX)$
	such that for all $i =1, \dots, n$,
	\[
	X_i = f_i(X_{\pa_i}, U_i) \quad \mbox{and} \quad
	Y_i = g_i(Y_{\pa_i}, V_i)
	\]
	for measurable functions $f_i, g_i  \colon \cX_{\pa_i} \times \bR \to \cX_i$
	and $\bR$-valued random variables $U_i, V_i$ satisfying
	\begin{itemize}
		\item[{\rm (i)}] $f_i$ is $L_i$-Lipschitz for a constant $L_i \ge 0$,
		\item[{\rm (ii)}] $U_1, \dots, U_n$ are independent and
		\item[{\rm (iii)}] $V_1, \dots, V_n$ are independent.
	\end{itemize}
	Then there exists a constant $C \ge 0$ depending on $G$ and
	$L_1, \dots, L_n$ such that
	\be \label{WGeps}
	W_{G, 1}(\mu, \nu) \leq C \sum_{i=1}^n \big\{\|f_i - g_i \|_{\infty} 
	+ W_1(\cL(U_i), \cL(V_i)) \big\},
	\ee
	where $\cL(U_i)$ and $\cL(V_i)$ are the distributions of $U_i$ and $V_i$, respectively.
\end{proposition}

\begin{proof}
	It is well-known that there exist independent pairs of random variables
	$(\tilde{U}_i, \tilde{V}_i)$ such that
	$\cL(\tilde{U}_i) = \cL(U_i)$, $\cL(\tilde{V}_i) = \cL(V_i)$ and
	\[
	\bE \, |\tilde{U}_i - \tilde{V}_i| = W_1(\cL(U_i), \cL(V_i))
	\quad \mbox{for all } i = 1, \dots, n;
	\]
	see e.g. \citep{villani2009optimal}. 
	Let us define iteratively
	\be \label{tildeXY}
	(\tilde{X}_i, \tilde{Y}_i) = (f_i(\tilde{X}_{\pa_i}, \tilde{U}_i), 
	g_i(\tilde{Y}_{\pa_i}, \tilde{V}_i)), \quad 
	i = 1, \dots, n.
	\ee
	Then, $\tilde{X} \sim X$ and $\tilde{Y} \sim Y$. Moreover, it can be seen from 
	\eqref{tildeXY} that 		\[
	(\tilde{X}_i, \tilde{Y}_i) \ind_{\tilde{X}_{\pa_i}, \tilde{Y}_{\pa_i}} 
	(\tilde{X}_{1:i-1}, \tilde{Y}_{1:i-1}), \quad
	\tilde{X}_i \ind_{\tilde{X}_{\pa_i}} \tilde{Y}_{\pa_i} \quad \mbox{and} \quad
	\tilde{Y}_i \ind_{\tilde{Y}_{\pa_i}} \tilde{X}_{\pa_i}
	\quad \mbox{for all } i =2, \dots, n.
	\]
	So it follows from Corollary \ref{cor:main} that the distribution $\pi$ of
	$(\tilde{X}, \tilde{Y})$ belongs to $\Pi^{\rm bc}_G(\mu, \nu)$.
	
	If we can show that 
	\be \label{Ci}
	\bE \, d_{\cX_i}(\tilde{X}_i, \tilde{Y}_i)
	\leq C_i \sum_{j=1}^i \big\{ \|f_j - g_j \|_{\infty} 
	+ W_1(\cL(U_j), \cL(V_j)) \big\}
	\ee
	for constants $C_i \ge 0$ depending on $G$ and $L_1, \dots, L_i$,
	we obtain \eqref{WGeps} with $C = \sum_{i=1}^n C_i$. 	
	We prove \eqref{Ci} by induction over $i = 1, \dots, n$. First, note that 
	\beas
	&& \bE \, d_{\cX_1}(\tilde{X}_1, \tilde{Y}_1) \leq 
	\bE \edg{d_{\cX_1}(f_1(\tilde{U}_1), f_1(\tilde{V}_1)) 
		+ d_{\cX_1} (f_1(\tilde{V}_1), g_1(\tilde{V}_1)}\\
	&\le& L_1 \, \bE |\tilde{U}_1 - \tilde{V}_1| 
	+ \|f_1 - g_1\|_{\infty} 
	= L_1 \, W_1(\cL(U_1), \cL(V_1))
	+ \|f_1 - g_1\|_{\infty},
	\eeas
	showing that for $i = 1$, \eqref{Ci} holds with $C_1 = \max \crl{L_1,1}$.
	For $i \ge 2$, assuming that \eqref{Ci} holds for all $j < i$, one obtains
	\beas
	&& \bE \, d_{\cX_i}(\tilde{X}_i, \tilde{Y}_i) \leq 
	\bE \edg{d_{\cX_i}(f_i(\tilde{X}_{\pa_i}, \tilde{U}_i), f_i(\tilde{Y}_{\pa_i}, \tilde{V}_i)) 
		+ d_{\cX_i} (f_i(\tilde{Y}_{\pa_i}, \tilde{V}_i), g_i(\tilde{Y}_{\pa_i}, \tilde{V}_i)} \\
	&\leq& L_i\, \bE \edg{\sum_{j \in \pa_i} d_{\cX_j}(\tilde{X}_j, \tilde{Y}_j) + |\tilde{U}_i - \tilde{V}_i| }
	+ \|f_i - g_i\|_\infty\\
	&\le& L_i \sum_{j \in \pa_i} C_j \sum_{k=1}^{j} \big\{ \|f_k - g_k \|_{\infty} 
	+ W_1(\cL(U_k), \cL(V_k)) \big\}
	+ L_i  W_1(\cL(U_i), \cL(V_i)) + \|f_i - g_i\|_{\infty}.
	\eeas
	So \eqref{Ci} holds with 
	\[
	C_i = \max \crl{L_i \sum_{j \in \pa_i} C_j \, ,\, L_i \, , \, 1},
	\quad i = 1, \dots, n,
	\] 
	and the proposition follows.
\end{proof}

\subsection{$G$-causal Wasserstein interpolation}\label{subsec:geodesic} 

In this section, we assume that in addition to a metric $d_{\cX}$ compatible 
with the Polish topology, $\cX$ is endowed with a vector space structure 
so that the vector space operations are continuous.

Note that for a sorted DAG $G = (V,E)$, $p \in [1, \infty)$ and given measures 
$\mu, \nu \in \cP_{G,p}(\cX)$,
there exists a sequence $(\pi^k)_{k \ge 1}$ of 
couplings in $\Pi^{\rm bc}_{G}(\mu, \nu)$ such that
\[
\int_{\cX \times \cX} d_{\cX}(x,y)^p d\pi^k(dx, dy) \to W_{G,p}^p(\mu, \nu) \quad
\mbox{for } k \to \infty.
\]
Since $\Pi^{\rm bc}_{G}(\mu, \nu) \subseteq \Pi(\mu, \nu)$, and the latter
is a weakly compact subset of
$\cP(\cX \times \cX)$, there exists a subsequence, again denoted $(\pi^k)_{k \ge 1}$, 
which weakly converges to a measure $\pi \in \Pi(\mu, \nu)$. This leads us to the 
following notion of $G$-causal interpolation, building on the concept of standard Wasserstein geodesics (cf.~\citep{villani2009optimal}) and more recent work studying causal Wasserstein interpolation between distributions of stochastic processes; see, e.g. ~\citep[Section 5.4]{bartl2021wasserstein}). We aim for a $G$-causal interpolation concept which, under suitable assumptions, ensures  that the interpolation between two $G$-compatible distributions preserves $G$-compatibility.

\begin{definition}\label{def:interpol}
Let $G = (V,E)$ be a sorted DAG and $p \in [1, \infty)$.
For given $\mu, \nu \in \cP_{G,p}(\cX)$, let $(\pi^k)_{k \ge 1}$ 
be a sequence of couplings in $\Pi^{\rm bc}_{G}(\mu, \nu)$ such that 
	\[
	\int_{\cX \times \cX} d_{\cX}(x,y)^p d\pi^k(dx, dy) \to W_{G,p}^p(\mu, \nu) \quad 
	\mbox{for } k \to \infty,
	\]
	and $\pi^k$ weakly converges to some $\pi \in \Pi(\mu, \nu)$. For $\lambda \in [0,1]$, 
	denote by $\kappa_{\lambda}$ the distribution of \linebreak
	$(1-\lambda) X + \lambda Y$, where $(X, Y) \sim \pi$. Then we call
	$(\kappa_\lambda)_{\lambda \in [0, 1]}$ a $W_{G, p}$-interpolation between $\mu$ and $\nu$.
\end{definition}

The measure $\pi$ in Definition \ref{def:interpol} 
is a limit of measures $\pi^k \in \Pi^{\rm bc}_{G}(\mu, \nu)$.
But in general, it might not belong to $\Pi^{\rm bc}_{G}(\mu, \nu)$ 
(see Proposition \ref{prop:basic_prop}.(iv)), and even if 
it does, the measures $\kappa_{\lambda}$ are not necessarily $G$-compatible
(see Example \ref{ex:Markov} below). Nevertheless, there are situations where the interpolations $\kappa_\lambda$ are $G$-compatible, which we study below.

The following result gives conditions under which the distribution of a convex 
combination of the form $(1-\lambda)X + \lambda Y$ is $G$-compatible. 

\begin{proposition} \label{prop:Gcomp}
	Let $G = (V,E)$ be a sorted DAG and $\mu, \nu \in \cP_{G}(\cX)$. 
	Assume $(X,Y) \sim \pi$ for a distribution $\pi \in \Pi^{\rm bc}_G(\mu, \nu)$. 
	Then, for any $\lambda \in [0,1]$, the distribution of $(X, Y, (1-\lambda)X + \lambda Y)$ is 
	$G$-compatible. Moreover, if there exists a measurable set $A \subseteq \cX \times \cX$
	with $\pi(A) = 1$ such that for each $i = 2, \dots, n$, the map
	\[
	h_{\lambda, i} : \cX_{\pa_i} \times \cX_{\pa_i} \rightarrow \cX_{\pa_i}, 
	~(x_{\pa_i}, y_{\pa_i}) \mapsto (1-\lambda) x_{\pa_i} + \lambda y_{\pa_i}
	\] 
	is injective on the projection of $A$ to $\cX_{\pa_i} \times \cX_{\pa_i}$,
	then the distribution of $\lambda X + (1-\lambda)Y$ is $G$-compatible.
\end{proposition}

\begin{proof}
	Since $\pi$ is in $\Pi^{\rm bc}_G(\mu, \nu)$, we obtain from Corollary \ref{cor:main} that
	\[
	(X_i, Y_i) \ind_{X_{\pa_i}, Y_{\pa_i}} (X_{1:i-1}, Y_{1:i-1}), 
	\]
	which implies
	\be \label{XYZ}
	(X_i, Y_i, Z_{\lambda, i}) 
	\ind_{X_{\pa_i}, Y_{\pa_i}, Z_{\lambda, \pa_i}} (X_{1:i-1}, Y_{1:i-1}, Z_{\lambda, 1:i-1})
	\ee
	for all $\lambda \in [0,1]$ and $Z_{\lambda} = (1-\lambda) X + \lambda Y$. By
	Definition \ref{def:Gcomp}, this shows that the distribution of $(X,Y, Z_{\lambda})$ 
	is $G$-compatible.
	
	Now, let us assume there exists a measurable set $A \subseteq \cX \times \cX$
	with $\pi(A) = 1$ such that for all $i = 2, \dots, n$, 
	$h_{\lambda, i}$ is injective on the projection of $A$ to $\cX_{\pa_i} \times \cX_{\pa_i}$.
	If $(\Omega, {\cal F}, \bP)$ is the probability space on which
	$(X,Y)$ is defined, we can, without loss of generality, assume 
	that $\Omega = (X,Y)^{-1}(A)$ since restricting $(X,Y)$ to
	$(X,Y)^{-1}(A)$ does not change its distribution.
	But then, $(X_{\pa_i}(\omega), Y_{\pa_i}(\omega)) \mapsto 
	(1-\lambda) X_{\pa_i}(\omega) + \lambda Y_{\pa_i}(\omega)$ 
	is injective for all $\omega \in \Omega$ and $i=2, \dots, n$. It follows that 
	$\sigma(X_{\pa_i}, Y_{\pa_i}) = \sigma(Z_{\lambda, \pa_i})$, which together with 
	\eqref{XYZ} gives $Z_{\lambda, i} \ind_{Z_{\lambda, \pa_i}} Z_{\lambda, 1:i-1}$
	for all $i = 2, \dots, n$. This shows that $Z_{\lambda}$ satisfies condition (ii) 
	of Definition \ref{def:Gcomp} and therefore, is $G$-compatible.
\end{proof}

\begin{corollary} \label{cor:finmarg}
	Let $G = (V,E)$ be a sorted DAG and $\mu, \nu \in \cP_{G}(\cX)$. 
	Assume $(X,Y) \sim \pi$ for a distribution $\pi \in \Pi^{\rm bc}_G(\mu, \nu)$. 
	\begin{itemize}
		\item[{\rm (i)}] If $\mu$ and $\nu$ are finitely supported, the distribution 
		of $(1-\lambda) X + \lambda Y$ is $G$-compatible for all but finitely 
		many $\lambda \in [0,1]$.
		\item[{\rm (ii)}] If $\mu$ and $\nu$ are countably supported, the distribution 
		of $(1-\lambda) X + \lambda Y$ is $G$-compatible for all but countably 
		many $\lambda \in [0,1]$.
	\end{itemize}
\end{corollary}

\begin{proof}
	We show (i). The proof of (ii) is analogous. Assume 
	$\mu$ and $\nu$ are of the form $\mu = \sum_{j=1}^J p_j \delta_{x_i}$ and 
	$\nu = \sum_{j=1}^J q_i \delta_{y_j}$ for $p_j, q_j \ge 0$ and 
	$x_j, y_j \in \cX$, $j = 1, \dots, J$. Then, for all
	$i \ge 2$, $x,x' \in \{x_1, \dots, x_J\}$ and $y, y' \in \{y_1, \dots, y_J\}$ such that 
	$(x_{\pa_i}, y_{\pa_i}) \neq (x'_{\pa_i}, y'_{\pa_i})$, there exists at most one
	$\lambda \in [0,1]$ satisfying
	\[
	\begin{aligned}
	&(1-\lambda) x_{\pa_i} + \lambda y_{\pa_i} 
	= (1-\lambda) x'_{\pa_i} + \lambda y'_{\pa_i} \\
	\Leftrightarrow & \quad 
	(1-\lambda) (x_{\pa_i} - x'_{\pa_i}) =
	\lambda (y'_{\pa_i} - y_{\pa_i} ).
	\end{aligned}
	\]
	It follows that for every $i = 2, \dots, n$ and all but finitely many $\lambda \in [0,1]$, 
	the map
	\[
	h_{\lambda, i} : \cX_{\pa_i} \times \cX_{\pa_i} \rightarrow \cX_{\pa_i}, 
	~(x_{\pa_i}, y_{\pa_i}) \mapsto (1-\lambda) x_{\pa_i} + \lambda y_{\pa_i}
	\] 
	is injective on the projection of $A = \{x_1, \dots, x_J\} \times \{y_1, \dots, y_J\}$ 
	to $\cX_{\pa_i} \times \cX_{\pa_i}$. Since $\pi \in \Pi^{\rm bc}_G(\mu, \nu)$,
	we have $\pi(A) = 1$ and therefore, obtain from Proposition \ref{prop:Gcomp} that the distribution of 
	$\lambda X + (1-\lambda)Y$ is $G$-compatible.
\end{proof}

\begin{corollary}\label{cor:interpolation}
	Let $G = (V,E)$ be a sorted DAG, $p \in [1,\infty)$ and $\mu, \nu$ two finitely supported 
	measures in $\cP_{G}(\cX)$ (and therefore also in $\cP_{G,p}(\cX)$). Then 
	there exists a measure $\pi \in \Pi^{\rm bc}_G(\mu, \nu)$ such that 
	\be \label{minpi}
	\int_{\cX \times \cX} d_{\cX}(x,y)^p d\pi(dx, dy) = W_{G,p}^p(\mu, \nu).
	\ee
	Moreover, for $(X,Y) \sim \pi$, the distribution of $(1-\lambda)X + \lambda Y$
	is $G$-compatible for all but finitely many $\lambda \in [0,1]$.
\end{corollary}

\begin{proof}
	That there is a measure $\pi \in \Pi^{\rm bc}_G(\mu, \nu)$ satisfying \eqref{minpi} 
	follows directly from Proposition \ref{prop:wassersteinprop}.(iii). Moreover, for $(X,Y) \sim \pi$, 
	we obtain from Corollary \ref{cor:finmarg}.(i) that the distribution of $(1-\lambda)X + \lambda Y$
	is $G$-compatible for all but finitely many $\lambda \in [0,1]$.
\end{proof}

\begin{example} \label{ex:Markov}
	Consider the Markovian graph $G = (V,E)$ with $V = \{1, 2, 3\}$ and $E = \{(1,2), (2,3)\}$,
	and take $\cX = \bR \times \bR \times \bR$ endowed with the Euclidean distance. The 
	distributions $\mu = \frac{1}{2}(\delta_{(0, 0, 0)} + \delta_{(1, 1, 1)})$ and 
	$\nu = \frac{1}{2}(\delta_{(0, 1, 0)} + \delta_{(1, 0, 1)})$ are both Markovian, and 
	clearly, the optimal $W_2$-transport plan transports $(0, 0, 0)$ to $(0, 1, 0)$ 
	and $(1, 1, 1)$ to $(1, 0, 1)$. It can directly be seen from Definition \ref{def:causal}
	that this transport plan belongs to $\Pi^{\rm bc}_G(\mu, \nu)$. Therefore,
	\[
	\kappa_{\lambda} = \frac{1}{2} \brak{\delta_{(0, \lambda, 0)} + \delta_{(1, 1- \lambda, 1)}}, 
	\quad \lambda \in [0,1],
	\]
	is the unique $W_{G,2}$-interpolation between $\mu$ and $\nu$. Note that
	$\kappa_{\lambda}$ is Markovian for all $\lambda \in [0, 1]$ except for 
	$\lambda = 1/2$. But for $\lambda = 1/2$, it is still a weak limit of 
	Markovian distributions.
\end{example}

\begin{figure}
	\includegraphics[width=0.32\textwidth]{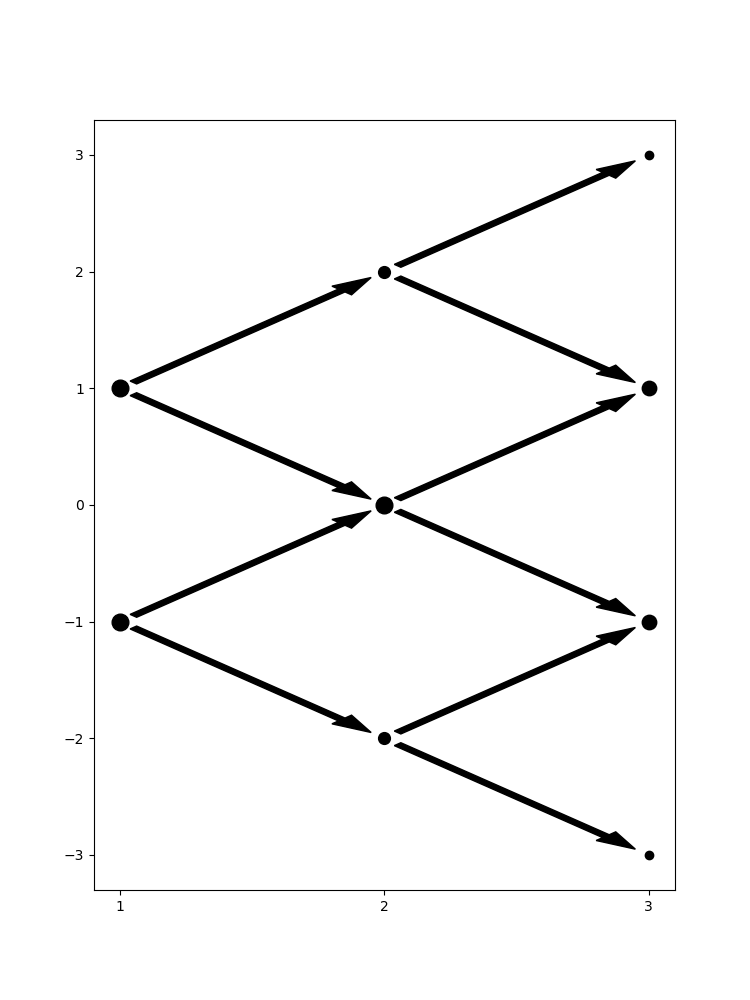}
	\includegraphics[width=0.32\textwidth]{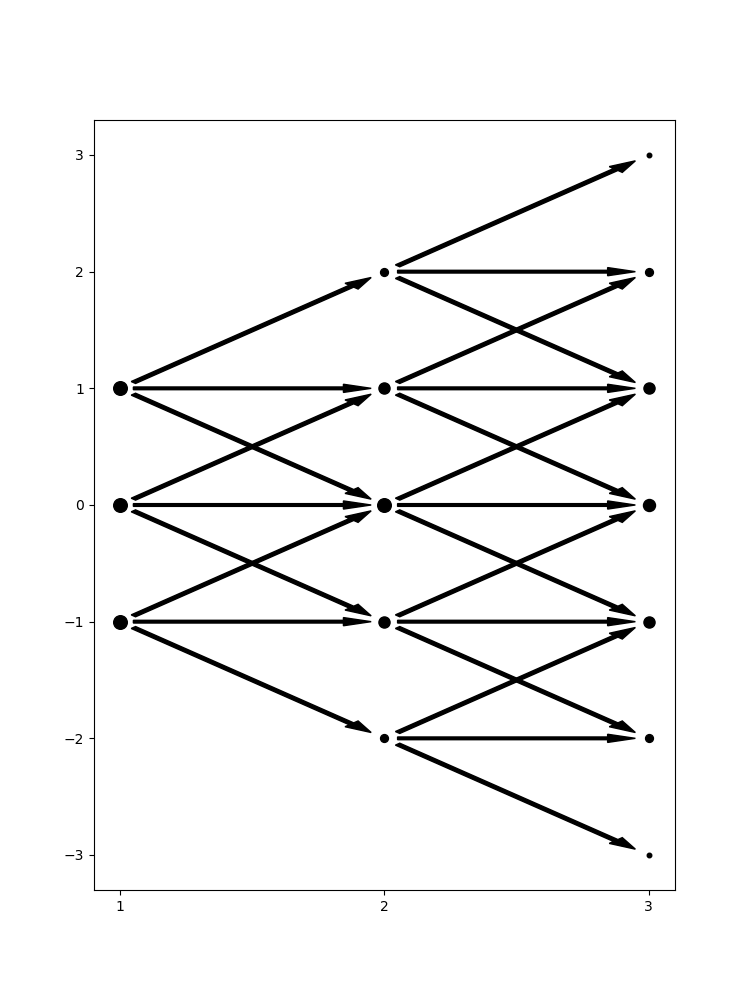}
	\includegraphics[width=0.32\textwidth]{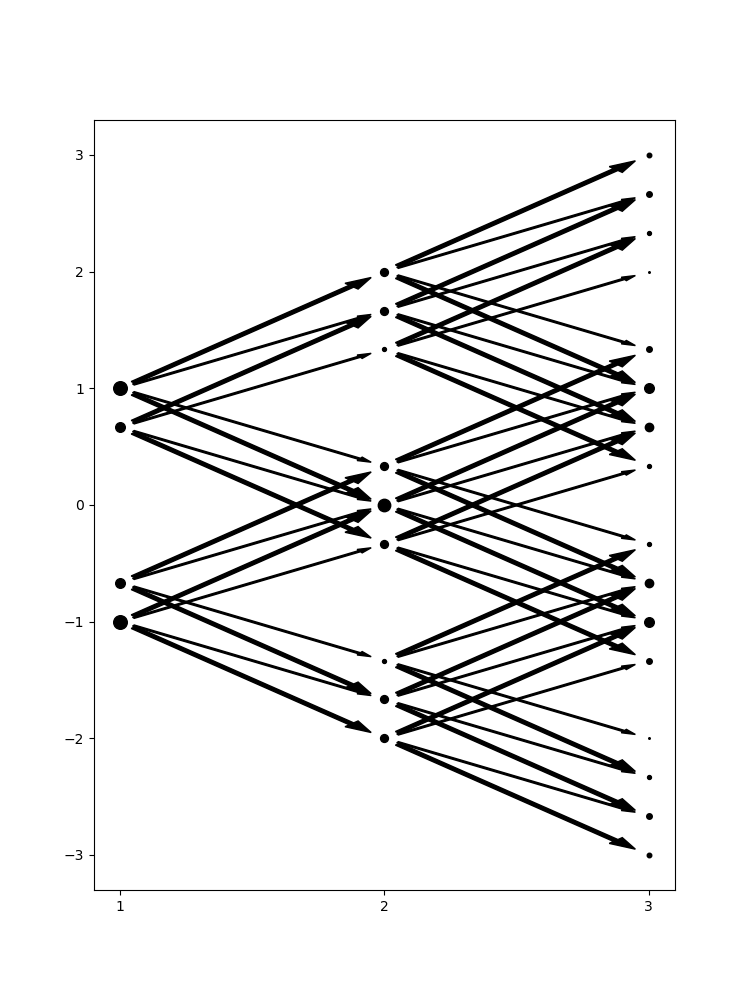}
	\caption{The binomial model $\mu$ on the left, the trinomial model $\nu$ in the middle, and the interpolating measure $\kappa_{1/3}$ under $W_{G, 2}$ on the right as given by Corollary \ref{cor:interpolation} using the Euclidean distance on $\mathbb{R}^3$ as the underlying metric.}
	\label{fig:binomtrinom}
\end{figure}

Next, we study an example where standard Wasserstein interpolation between 
two $G$-compatible measures leads to distributions that are far from $G$-compatible.

\begin{example} \label{ex:interpol}
	Consider the same setup as in Example \ref{ex:Markov}. Denote by $\mu$ the distribution 
	of the random walk $(X_1, X_2, X_3)$ given by $X_i = \sum_{j=1}^i \xi_j$ for i.i.d.\ innovations
	$\xi_j$ with distribution $\bP[\xi_j = \pm 1] = 1/2$ and $\nu$ the distribution of 
	$(Y_1, Y_2, Y_3)$ given by $Y_i = \sum_{j=1}^i \eta_j$ for i.i.d.\ innovations
	$\eta_j$ with distribution $\bP[\eta_j = -1,0,1] = 1/3$. The first two pictures in
	Figure \ref{fig:binomtrinom} show all possible trajectories of $X = (X_1, X_2, X_3)$
	and $Y = (Y_1, Y_2, Y_3)$, respectively.
	
	For the standard Wasserstein distance $W_2(\mu, \nu)$, a numerically obtained optimal coupling $\pi$ can be seen to satisfy 
	\begin{align*}
	\pi\big((X_3, Y_3) = (1, 0) \mid (X_1, Y_1) = (-1, 0), (X_2, Y_2) = (0, 0)\big) \\ < \pi\big((X_3, Y_3) = (1, 0) \mid (X_1, Y_1) = (1, 0), (X_2, Y_2) = (0, 0) \big),
	\end{align*}
	showing that $(X_i, Y_i)_{i=1, 2, 3}$ is not Markovian under $\pi$. It follows that, for all 
	but finitely many $\lambda \in [0,1]$, the resulting interpolations $\kappa_{\lambda}$ 
	under this coupling $\pi$ are also not Markovian.\footnote{This follows from the 
		same arguments as in the proofs of Proposition \ref{prop:Gcomp} and 
		Corollary \ref{cor:finmarg}, where the main argument is that for $Z^\lambda_k = (1-\lambda) X_k + \lambda Y_k$, one has $\sigma(Z^\lambda_k) = \sigma(X_k, Y_k)$ for all but finitely many $\lambda \in [0, 1]$.}
	
	On the other hand, using a $W_{G, p}$-interpolation as given by Corollary \ref{cor:interpolation}, the resulting interpolating measures $\kappa_\lambda$ are Markovian for all but finitely 
	many $\lambda \in [0, 1]$. The last picture in Figure \ref{fig:binomtrinom} shows 
	an illustration of the numerically computed $G$-causal (i.e., Markov) Wasserstein interpolation
	$\kappa_{1/3}$ using $W_{G, 1}$. 
\end{example}

\begin{appendix}
\section{Stability of conditional independence}
\label{app:stability}

\begin{lemma} \label{lemma:stable}
	Let $\cS, \cW, \cZ$ be measurable spaces and
	$S, W, Z$ the projections from $\cS \times \cW \times \cZ$
	to $\cS$, $\cW$, $\cZ$, respectively. Denote by $\cP(\cS \times \cW \times \cZ)$ the 
	collection of all probability measures on $\cS \times \cW \times \cZ$ equipped with 
	the product $\sigma$-algebra. Then the set
	\be \label{set}
	\{\pi \in \cP(\cS \times \cW \times \cZ) : S \ind_Z W \mbox{ with respect to } \pi \}
	\ee
	is closed under convergence in total variation.
\end{lemma}

\begin{proof}
	Let $(\pi^k)_{k \in \bN}$ be a sequence in \eqref{set} converging to a measure 
	$\pi \in \cP(\cS \times \cW \times \cZ)$ in total variation. By \eqref{condindep2}, one has
	\be \label{k=}
	\pi^k[A \times \cW \times \cZ \mid W, Z] = \pi^k[A \times \cW \times \cZ \mid Z] \quad \mbox{$\pi^k$-a.s.}
	\ee
	for every measurable subset $A \subseteq \cS$ and all $k \in \bN$.
	Let us denote 
	\[
	H_k = \pi^k[A \times \cW \times \cZ \mid Z] \quad \mbox{and} \quad
	H = \pi[A \times \cW \times \cZ \mid Z].
	\]
	If $\bE^k$ denotes expectation with respect to $\pi^k$ and $\bE$ 
	expectation with respect to $\pi$, one has
	\[
	\abs{\bE^k U - \bE \, U} =
	\abs{\int_0^1 \brak{\pi^k[U > u] - \pi[U > u]} du }
	\le \int_0^1 \abs{\pi^k[U > u] - \pi[U > u]} du
	\]
	for all random variables $U \colon \cS \times \cW \times \cZ \to [0,1]$,
	which shows that for every $\varepsilon > 0$, there exists a $k_0 \in \bN$ such that 
	\[
	\abs{\bE^k U - \bE \, U} \le \varepsilon \quad 
	\mbox{for every } k \ge k_0 \mbox{ and all such $U$.}
	\]
	In particular, 
	\[
	\abs{\bE^k (1_{A \times \cW \times \cZ} - H)^2 
		- \bE \, (1_{A \times \cW \times \cZ} - H)^2} \le \varepsilon
	\]
	and
	\[
	\abs{\bE^k (1_{A \times \cW \times \cZ} - H_k)^2 
		- \bE \, (1_{A \times \cW \times \cZ} - H_k)^2} \le \varepsilon
	\]
	for all $k \ge k_0$.
	Moreover, since the $H_k$ can be viewed as $L^2$-projections of 
	$1_{A \times \cW \times \cZ}$ on the space of square-integrable $Z$-measurable 
	random variables, one has
	\[
	\bE^k (1_{A \times \cW \times \cZ}  - H_k)^2 \le \bE^k(1_{A \times \cW \times \cZ}  - H)^2.
	\]
	Together with Pythagoras's theorem this gives
	\beas
	\bE (H_k - H)^2 &=& \bE (1_{A \times \cW \times \cZ}  - H_k)^2 
	- \bE (1_{A \times \cW \times \cZ}  - H)^2\\
	&\le& \bE^k (1_{A \times \cW \times \cZ}  - H_k)^2 + \varepsilon
	- \bE (1_{A \times \cW \times \cZ}  - H)^2\\
	&\le& \bE^k (1_{A \times \cW \times \cZ}  - H)^2 + \varepsilon
	- \bE (1_{A \times \cW \times \cZ}  - H)^2\\
	&\le& \bE (1_{A \times \cW \times \cZ}  - H)^2 + 2 \varepsilon
	- \bE (1_{A \times \cW \times \cZ}  - H)^2
	= 2 \varepsilon
	\quad \mbox{for all } k \ge k_0, 
	\eeas
	showing that $H_k \to H$ in $L^2(\pi)$. Analogously, one obtains that 
	$\pi^k[A \times \cW \times \cZ \mid W, Z]$ converges to 
	$\pi[A \times \cW \times \cZ \mid W, Z]$ in $L^2(\pi)$.
	So, equation \eqref{k=} is stable under convergence in total variation, 
	which proves the lemma.
\end{proof}
\section{Counterexample to the triangle inequality}
\label{sec:counter}

In this section, we provide an example showing that in general, 
$W_{G, 1}$ does not satisfy the triangle inequality. 
We consider the Markovian graph $G = (V,E)$ with vertices $V = \{1,2,3\}$ and edges 
$\{(1,2), (2,3) \}$ and choose $\cX = \cX_1 \times \cX_2 \times \cX_3$, where 
$\cX_1$, $\cX_2$, $\cX_3$ are all equal to the same abstract discrete 
Polish space $\crl{x_1, x_2, x_3, x_4}$. Define $\mu, \nu, \eta$ as follows, which we note are Markovian:
\begin{align*}
\mu &= \frac{1}{4} \left( \delta_{(x_1, x_1, x_1)} + \delta_{(x_1, x_1, x_2)} 
+ \delta_{(x_2, x_1, x_1)} + \delta_{(x_2, x_1, x_2)} \right) \\
\nu &= \frac{1}{4} \left( \delta_{(x_1, x_2, x_1)} + \delta_{(x_1, x_2, x_2)} 
+ \delta_{(x_2, x_3, x_1)} + \delta_{(x_2, x_3, x_2)} \right)\\
\eta &= \frac{1}{4} \left( \delta_{(x_3, x_4, x_3)} + \delta_{(x_3, x_4, x_4)} 
+ \delta_{(x_4, x_4, x_3)} + \delta_{(x_4, x_4, x_4)} \right).
\end{align*}
To gain some intuition why $W_{G,1}$ violates the triangle inquality in this example, consider stochastic processes $(X_1, X_2, X_3) \sim \mu$, $(Y_1, Y_2, Y_3) \sim \nu$, $(Z_1, Z_2, Z_3) \sim \eta$ and note that $X_2$ and $Z_2$ are constant while $\sigma(Y_2) = \sigma(Y_1, Y_2)$. So it can be seen from Corollary \ref{cor:main} that transport plans in $\Pi^{\rm bc}_G(\mu, \eta$) are more restricted than those in $\Pi^{\rm bc}_G(\mu, \nu)$ and $\Pi^{\rm bc}_G(\nu, \eta)$, which, for a suitable metric on $\mathcal{X}$, implies that $W_{G, 1}(\mu, \eta)$ is large compared to $W_{G,1}(\mu, \nu)$ and $W_{G, 1}(\nu, \eta)$.


We specify the distances between the 12 sequences $(x_1,x_1, x_1), (x_1, x_1, x_2), \dots$
used to define the measures $\mu, \nu, \eta$ with the following $12 \times 12$ matrix $M$:
\setcounter{MaxMatrixCols}{12}
\[\begin{bmatrix}
0  & 0.53 & 1.08 & 1.33 & 1.29 & 0.78 & 0.64 & 0.44 & 1.15 & 1.3  & 1.92 & 1.38 \\
0.53&  0  & 0.7&  1.05 &1.18& 0.91 &0.11 &0.97 &1.1  &1.83 &1.9  &1.91 \\
1.08& 0.7  &0   &0.98 &1.11 &0.49 &0.59 &0.64 &0.86 &1.88 &1.82 &1.58 \\
1.33& 1.05 &0.98 &0   &0.13 &1.15 &0.94 &0.98 &1.52 &1.25 &1.87 &1.92 \\
1.29& 1.18 &1.11 &0.13 &0   &1.02 &1.07 &0.85 &1.39 &1.38 &2  &1.79 \\
0.78& 0.91 &0.49 &1.15 &1.02 &0   &1.02 &1.13 &0.37 &1.74 &1.77 &1.22 \\
0.64& 0.11 &0.59 &0.94 &1.07 &1.02 &0   &1.08 &1.21 &1.94 &1.79 &2.02 \\
0.44& 0.97 &0.64 &0.98 &0.85 &1.13 &1.08 &0   &1.5  &1.24 &1.86 &0.94 \\
1.15& 1.1  &0.86 &1.52 &1.39 &0.37 &1.21 &1.5  &0   &1.37 &1.4  &0.85 \\
1.3 & 1.83 &1.88 &1.25 &1.38 &1.74 &1.94 &1.24 &1.37 &0   &0.62 &1.2  \\
1.92& 1.9  &1.82 &1.87 &2   &1.77 &1.79 &1.86 &1.4  &0.62 &0   &1.69 \\
1.38& 1.91 &1.58 &1.92 &1.79 &1.22 &2.02 &0.94 &0.85 &1.2  &1.69 &0  
\end{bmatrix} .
\]
That is, $M_{1,2}$ is the distance between $(x_1,x_1, x_1)$ and $(x_1, x_1, x_2)$, 
$M_{1,3}$ is the distance between $(x_1, x_1, x_1)$ and $(x_2, x_1, x_1)$ and so on. 
We obtained $M$ by simulating symmetric matrices with zeros on the diagonal and 
positive entries off the diagonal and iteratively updating the entries as long as the 
triangle inequality was violated. More precisely, we started with a symmetric random matrix
$M^0$ with zeros on the diagonal and positive off-diagonal entries. Then we iteratively set
\[
M^k_{i,j} = 
\min\left\{M^{k-1}_{i, j}, M^{k-1}_{i,1} + M^{k-1}_{1,j}, M^{k-1}_{i,2} + M^{k-1}_{2,j}, \dots, 
M^{k-1}_{i,12} + M^{k-1}_{12,j}\right\}
\]
until $M^{k} = M^{k-1}$. This guarantees that $M$ defines a metric on 
$(x_1,x_1, x_1), (x_1, x_1, x_2) \dots$ From there it can be extended to a metric 
on $\cX$ by Fr\'echet embedding. Finally, we used Gurobi \citep{gurobi} to compute 
$W_{G, 1}(\mu, \nu) =  0.585$, $W_{G, 1}(\nu, \eta) = 2.24$ and
$W_{G, 1}(\mu, \eta) = 2.925$, showing that 
\[
W_{G, 1}(\mu, \eta) >  W_{G, 1}(\mu, \nu) + W_{G, 1}(\nu, \eta).
\]
\end{appendix}

\end{document}